\documentclass[11pt]{amsart}
\usepackage{psfrag}
\usepackage[parfill]{parskip}
\usepackage{float, graphicx}
\usepackage[]{epsfig}
\usepackage{amsmath, amsthm, amssymb}
\usepackage{epsfig}
\usepackage{verbatim}
\usepackage{multicol}
\usepackage{url}
\usepackage{latexsym}
\usepackage{mathrsfs}
\usepackage[colorlinks, bookmarks=true]{hyperref}
\usepackage{graphicx}
\usepackage{amsmath}
\usepackage{enumerate}
\usepackage[normalem]{ulem}

\newtheorem{thm}{Theorem}[section]
\newtheorem{theorem}{Theorem}[section]
\newtheorem{lem}[thm]{Lemma}

\newtheorem{prop}[thm]{Proposition}
\newtheorem{conj}[thm]{Conjecture}
\newtheorem{cor}[thm]{Corollary}
\newtheorem{corollary}[thm]{Corollary}

\newtheorem{exam}[thm]{Example}
\newtheorem{example}[thm]{Example}
\theoremstyle{definition}
\newtheorem{df}[thm]{Definition}
\newtheorem{definition}[thm]{Definition}
\theoremstyle{remark}
\newtheorem{rem}[thm]{Remark}
\numberwithin{equation}{section}

\title{Laurent Phenomenon Sequences}
\author{Joshua Alman, Cesar Cuenca, Jiaoyang Huang}

\address{Department of  Mathematics, 
MIT, Cambridge, MA, 02139}
\email{jalman@mit.edu}
\email{cuenca@mit.edu}
\email{jiaoyang@mit.edu}

\date{\today}

\begin{document}
\maketitle

\newcommand{\C}{\mathbb C}
\newcommand{\R}{\mathbb R}
\newcommand{\Z}{\mathbb Z}
\newcommand{\Q}{\mathbb Q}
\newcommand{\N}{\mathbb N}
\newcommand{\F}{\mathcal F}
\newcommand{\PP}{\mathcal P}

\begin{abstract}
In this paper, we undertake a systematic study of recurrences $x_{m+n}x_m = P(x_{m+1}, \ldots, x_{m+n-1})$ which exhibit the Laurent phenomenon. Some of the most famous among these sequences come from the Somos and the Gale-Robinson recurrences. Our approach is based on finding period 1 seeds of Laurent phenomenon algebras of Lam-Pylyavskyy. We completely classify polynomials $P$ that generate period 1 seeds in the cases of $n=2,3$ and of mutual binomial seeds. We also find several other interesting families of polynomials $P$ whose generated sequences exhibit the Laurent phenomenon. Our classification for binomial seeds is a direct generalization of a result by Fordy and Marsh, that employs a new combinatorial gadget we call a double quiver.
\end{abstract}

\section{Introduction}\label{intro}
The goal of this paper is to understand the Laurent phenomenon (\cite{fz}) appearing in Somos type recurrence relations, i.e., sequences $x_0,x_1,x_2,\ldots$ defined by recurrences of the form
\begin{align} \label{recurr}
x_mx_{m+n}=P(x_{m+1},x_{m+2},\ldots, x_{m+n-1}), \qquad m=0,1,2,\ldots
\end{align}
where $P$ is a polynomial. The prototypical example of such a sequence is the Somos-$n$ sequence, given by the recurrence
\begin{align*}
x_mx_{m+n}=\sum_{1 \leq i \leq \frac{n}{2}} x_{m+i}x_{m+n-i}
\end{align*}
Another key example is the Gale-Robinson sequence, which is given by the recurrence
\begin{align} \label{GR}
x_{m+n}x_{m}=\alpha x_{m+r}x_{m+n-r}+\beta
x_{m+p}x_{m+n-p}+\gamma x_{m+q}x_{m+n-q},
\end{align}
for some $p,q,r>0$ with $p+q+r=n$.

It is clear that terms of these sequences can be written as rational functions of the first $n$ terms. Remarkably, in a Somos-$n$ sequence, for $1 \leq n \leq 7$, or any Gale-Robinson sequence, each term can, in fact, be written as a Laurent polynomial in the first $n$ terms. This Laurent phenomenon for Gale-Robinson sequences was first proven in \cite{fz}. They initated a study of Somos-type recurrences related to cluster algebras.

In \cite{fm}, these types of sequences are studied as exchange relations in cluster mutation-periodic quivers. If mutating at a vertex $v$ in a quiver $Q$ results in a rotation of $Q$, then the exchange polynomial $P$ associated with $v$ yields a sequence of the form (\ref{recurr}) exhibiting the Laurent phenomenon. It is found in \cite{fm} that 2-term Gale Robinson sequences, which are of the form (\ref{GR}) with $\gamma=0$, are exactly the polynomials $P$ we can obtain in this way.

However, there are limitations to what the exchange polynomials can be in a cluster algebra; they need to be binomials which come from a quiver. Gale-Robinson sequences and other natural examples show that the Laurent phenomenon can hold when $P$ has other forms. Lam and Pylyavskyy introduced in \cite{lp} a generalization of cluster algebras that removes these constraints called Laurent phenomenon algebras, or LP algebras. They showed that LP algebras exhibit the Laurent phenomenon, which will imply that a period 1 LP algebra yields a Laurent phenomenon sequence.

In this paper, we study LP algebras arising from period 1 seeds to find more far-reaching results than \cite{fm}. We prove classification results for period 1 seeds when $n=2,3$. For $n=2$, we find, by comparing our seeds with the classification of Laurent phenomenon sequences by Speyer and Musiker \cite{musiker}, that a polynomial generates a period 1 LP seed if and only if it generates a Laurent phenomenon sequence. We also classify mutual binomial period 1 seeds. As exchange polynomials in a cluster algebra are all binomials, this result will generalize the classification theorem in \cite{fm} by taking advantage of the lessened constraints of LP algebras. In fact, our classification is described using a generalization of quivers that we introduce, called double quivers, which operates within the machinery of LP algebras. We also give large families of polynomials that generate period 1 seeds. For many of these, to the best of our knowledge, the Laurent phenomenon had not been proven. Similar to \cite[\S 9]{fm}, we also investigate conserved quantities, k-invariants and (multi)linearizations of these families, which yields insights into their integrability.

The remainder of the paper is organized as follows. In Section 2, we introduce the relevant notions of LP algebras. We give an algorithm that, given a polynomial $P$, finds the unique candidate for a period 1 LP algebra, and would show that (\ref{recurr}) has the Laurent phenomenon. We also link to our implementation of the algorithm in Sage. In Section 3, we summarize our results. In Sections 4 and 5, we prove our results about mutual binomial seeds and small $n$, respectively. In Section 6, we give the period 1 seeds corresponding to our families of examples. Finally, in Section 7 we investigate the conserved quantities and integrability of some of these recurrences.

\section{Laurent Phenomenon Algebras}\label{lpalgebras}

Before we state our main results, we introduce period 1 Laurent phenomenon algebra seeds and their important properties.

\subsection{Seeds and Mutations}
In this subsection, we define Laurent phenomenon (LP) algebras and related notions from \cite{lp}. 

Let $\F$ be the field of rational functions in $n$ independent variables over $\Q$.
A {\it{seed}} $t$ is a pair $\bf (x, P)$ where:

\begin{itemize}
    \item $\bf{x}$ $= \{x_0, \ldots, x_{n-1}\}$ is a trascendence basis for $\F$ over $\Q$.
	\item $\mathbf{P} = \{P_0, \ldots, P_{n-1}\}$ is a collection of polynomials in $\PP = \Z[x_0, \ldots, x_{n-1}]$ satisfying:
	
	(LP1) $P_i \in\PP$ is irreducible and is not divisible by any $x_j$.
	 
	(LP2) $P_i$ does not depend on $x_i$.
\end{itemize}

Equivalently, if we denote by $\PP_i$ the polynomials in $\PP$ that satisfy (LP1) and (LP2), then we say that a seed $\bf (x, P)$ consists of a collection of pairs $(x_i, P_i)$, $0\leq i\leq n-1$, such that $P_i\in\PP_i$ for all $i$.

Borrowing terminology from the theory of cluster algebras, the set $\{ x_0, \ldots, x_{n-1} \}$ is called a \emph{cluster}, each of $x_0, \ldots, x_{n-1}$ is called a \emph{cluster variable}, and the polynomials $P_0, \ldots, P_{n-1}$ are the associated \emph{exchange polynomials}. 

We now define mutation in LP algebras. For $k \in \{ 0, \ldots, n-1 \}$, we say that a seed $\bf(x', P')$ is obtained from $\bf (x, P)$ by {\it mutation at $k$}, which we denote $\mu_k(\mathbf{x}, \mathbf{P}) = (\mathbf{x'}, \mathbf{P'})$, if $t' = \bf(x', P')$ comes from $t = \bf(x, P)$ via the following sequence of steps:

\begin{enumerate}
\item Let $\mathcal{L}(t) = \Z[x_0 ^{\pm1}, \ldots, x_{n-1} ^{\pm1}]$ be the Laurent polynomial ring in the cluster variables. Define the \emph{exchange Laurent polynomials} $\{ \widehat{P}_0, \ldots, \widehat{P}_{n-1} \} \subset \mathcal{L}(t)$ to be the unique set of Laurent polynomials satisfying:

\begin{itemize}
\item For each $j \in \{0,\ldots, n-1\}$ there are $a_1, \ldots, a_{j-1},a_{j+1},\ldots,a_{n-1} \in \Z_{\leq 0}$ such that $\widehat{P}_j = x_1^{a_1} \ldots x_{j-1}^{a_{j-1}} x_{j+1}^{a_{j+1}}\ldots x_{n-1}^{a_{n-1}} P_j$

\item For each $i,j \in \{0,\ldots,n-1\}$ with $i \neq j$, if we let $\mathcal{L}_j(t) = \Z[x_0 ^{\pm1}, \ldots, x_{j-1}^{\pm 1}, x^{\pm 1}, x_{j+1}^{\pm 1}, \ldots x_{n-1} ^{\pm1}]$, then we have $\widehat{P}_i | _{x_j \leftarrow P_j/x} \in \mathcal{L}_j(t)$ and $\widehat{P}_i|_{x_j \leftarrow P_j/x}$ is not divisible by $P_j$ in $\mathcal{L}_j(t)$.
\end{itemize}

\item The new cluster $\mathbf{x'} = \{x'_0, \ldots, x'_{n-1}\}$ is given by:

$$x_i' = \begin{cases}
x_i &\text{if } i \neq k, \\
\widehat{P}_k/x_k &\text{if } i = k.
\end{cases}$$

\item Define the polynomial

$$G_j = P_j \left|_{x_k \leftarrow \frac{\widehat{P}_k|_{x_j \leftarrow 0}}{x_k'}.}\right. $$

\item Define $H_j$ to be the result of removing all common factors with $\widehat{P}_k|_{x_j \leftarrow 0}$ from $G_j$ (in the unique factorization domain $\Z[x_0, \ldots, x_{k-1}, x_{k+1}, \ldots, x_{j-1}, x_{j+1}, \ldots x_{n-1}]$).

\item Define the new exchange polynomial $P'_j = M_j H_j$, where $M_j$ is the unique Laurent monomial in $x'_0, \ldots, x'_{n-1}$ for which $M_j H_j$ is not divisible by any Laurent monomial.

\item The new seed is given by $\mu_k(\mathbf{x}, \mathbf{P}) = (\mathbf{x'}, \mathbf{P'}) = (\{x'_0, \ldots, x'_{n-1} \}, \{ P'_0, \ldots, P'_{n-1} \})$.
\end{enumerate}

\begin{rem}
If $P_j$ does not depend on $x_k$, then $G_j = H_j = P_j$ and $M_j = 1$, implying that $P_j' = P_j$.
\end{rem}

All the necessary existence and uniqueness conditions to show that the above mutation gives a unique valid seed can be found in \cite[\S 2]{lp}. We will often abuse notation and write:
$$\mu_0(\{x_0, \ldots, x_{n-1}\}, \{P_0, \ldots, P_{n-1}\}) = (\{x_1, \ldots, x_{n}\}, \{P'_1, \ldots, P'_{n}\}),$$
where $x_n = x'_0$ and $P_n = P'_0$, since $x'_i = x_i$ for all $0<i\leq n-1$. More generally, we will write:
$$\mu_i(\{x_i, \ldots, x_{n+i-1}\}, \{P_i, \ldots, P_{n+i-1}\}) = (\{x_{i+1}, \ldots, x_{n+i}\}, \{P'_{i+1}, \ldots, P'_{n+i}\}),$$

where $x_{n+i} = x_i'$ and $P_{n+i}' = P_i'$. For any seed $t$, the \emph{Laurent phenomenon algebra} $\mathcal{A}(t)$ is the commutative subring of $\Q$ generated by the cluster variables of the seeds that can be obtained after a finite sequence of mutations from $t$. We call $t$ the \emph{initial seed} of $\mathcal{A}(t)$. The importance of this definition has much to do with the next theorem.

\begin{theorem}{\cite[Theorem 5.1]{lp}} \label{LPLP}
Let $\mathcal{A}$ be a Laurent phenomenon algebra and $t = \bf(x, P)$ a seed of $\mathcal{A}$. If $\mathbf{x} = \{x_0, \ldots, x_{n-1}\}$, then any cluster variable of $\mathcal{A}$ belongs to the Laurent polynomial ring $\mathcal{L}(t) = \Z[x_0^{\pm 1}, \ldots, x_{n-1}^{\pm 1}]$.
\end{theorem}

\subsection{Period 1 Seeds}

In this paper, we are primarily interested in period 1 seeds that will be defined shortly.

\begin{definition}
For any polynomial $P \in \Q[x_{-1}, x_0, x_1, x_2, \ldots]$, the \emph{upshift} of $P$ is the polynomial $$Q = P|_{x_i \leftarrow x_{i+1} ~ \forall i} \in \Q[x_0, x_1, x_2, \ldots].$$
If $P$ does not depend on $x_{-1}$, then the \emph{downshift} is defined analogously: $$R = P|_{x_i \leftarrow x_{i-1} ~ \forall i} \in \Q[x_0, x_1, \ldots].$$
\end{definition}

\begin{example}
The upshift and downshift of $x_2x_3^2 + x_7 - 3x_9$
are $x_3x_4^2 + x_8 - 3x_{10}$ and $x_1x_2^2 + x_6 - 3x_8$, respectively.
\end{example}

\begin{definition}
Let $t = (\{ x_0, \ldots, x_{n-1}\}, \{P_0, \ldots, P_{n-1}\})$ be a seed and $\mu_0(t) = (\{ x_1, \ldots, x_{n}\}, \{P'_1, \ldots, P'_{n}\})$ be its mutation at $x_0$ ($x_n = x_0'$ and $P_n' = P_0'$).
Then $t$ is a {\em period 1 seed} if $P_i$ is the downshift of $P_{i+1}'$ for all $0 \leq i < n-1$, and $P_{n-1}$ is the downshift of $P'_0$ (or equivalently of $P_0$).
\end{definition}

Period 1 seeds are interesting in light of Theorem \ref{LPLP}, as they provide the machinery to prove that some recurrence sequences satisfy the Laurent phenomenon:

\begin{corollary} \label{laur}
Let $P \in \Z[x_1, \ldots, x_{n-1}]$ be any irreducible polynomial, not divisible by any $x_j$. If there exists a period 1 seed $t = (\{x_0, \ldots, x_{n-1} \}, \{ P_0, \ldots, P_{n-1}\})$ with $P_0 = P$, then the sequence $\{x_i\}_{i \geq 0}$ of rational functions of $x_0, \ldots, x_{n-1}$, defined by
$$x_{m+n} = \frac{P(x_{m+1}, \ldots, x_{m+n-1})}{x_m} \textrm{ for all } m\geq 0,$$
consists entirely of Laurent polynomials.
\end{corollary}
\begin{proof}
The sequence $\{x_i\}_{i \geq 0}$ consists of the cluster variables we get by applying the mutations $\mu_0, \mu_1, \ldots$ in order to $t$.
\end{proof}

\begin{example}
For $n = 3$, the polynomial $P = x_1x_2+1$ is in the period 1 seed
$\{(x_0, x_1x_2 + 1), (x_1, x_0 + x_2), (x_2, x_0x_1 + 1)\}$ and generates the sequence
\begin{align*}&x_0, x_1, x_2, \frac{x_1x_2 + 1}{x_0}, \frac{x_1x_2^2 + x_0 + x_2}{x_0x_1},
\frac{(x_1x_2 + 1)(x_1x_2^2 + x_0 + x_2) + x_0^2x_1}{x_0^2x_1},\\
&\frac{(x_1x_2^2 + x_0 + x_2)^2 + x_0^2x_1(x_0 + x_2)}{x_0^2x_1^2x_2}, \ldots \end{align*}
\end{example}

\begin{corollary}
If $P \in \Z[x_1, \ldots, x_{n-1}]$ satisfies the conditions in Corollary \ref{laur}, then the sequence $\{x_i\}_{i \geq 0}$ defined by
$$x_{m+n} = \frac{P(x_{m+1}, \ldots, x_{m+n-1})}{x_m} \textrm{ for all $m\geq 0$},$$
with initial values $x_i=1$, for all $0\leq i<n$,
consists entirely of integers. If the coefficients of $P$ are positive, then the sequence consists entirely of positive integers.
\end{corollary}

\begin{definition}
For a period 1 seed $t = (\{x_0, \ldots, x_{n-1} \}, \{ P_0, \ldots, P_{n-1}\})$, we will say that $t$ is the period 1 seed {\em generated by} $P_0$, or that $P_0$ {\em generates} $t$. We call $P_0$ a {\em period 1 polynomial}, or say that $P_0$ is {\em 1 periodic}.
\end{definition}

At this point, it is worth noting that the converse of Corollary \ref{laur} is not true. For instance, consider $n = 3$ and $P(x_1, x_2) = x_1 + x_2 + 1$. After reading the next section, the reader should be able to easily confirm that $P$ does not generate a period 1 seed. However, the sequence generated by $P$ is periodic and satisfies the Laurent phenomenon:
\begin{align*}&x_1, x_2, x_3, \frac{x_2 + x_3 + 1}{x_1}, \frac{x_1 x_3 + x_1 + x_2 + x_3 + 1}{x_0 x_1}, \frac{x_1 x_2 + x_2^2 + x_1 x_3 + x_2 x_3 + x_1 +2 x_2 + x_3 + 1}{x_1 x_2 x_3}, \\
&\frac{x_1 x_3 + x_1 + x_2 + x_3 + 1}{x_0 x_1}, \frac{x_2 + x_3 + 1}{x_1}, x_1, x_2, x_3, \ldots \end{align*}

\subsection{Generation of period 1 seeds}

In this section, we propose a method for obtaining period 1 seeds generated by particular polynomials, using a variant of the method in \cite{fz}.

Given a polynomial $P\in\PP_0 = \Z[x_1, \ldots, x_{n-1}]$, we define a map $\tau = \tau_{P} : \PP \to \PP$, which takes polynomials in
$\PP_i$ to polynomials in $\PP_{i-1}$, for all $i > 0$.
If $Q\in\PP_i$, then $\tau_P(Q)\in\PP_{i-1}$ is computed according to the following algorithm:

\begin{enumerate}
    \item Let $\displaystyle G = G(x_1, \ldots, \widehat{x}_i, \ldots, x_i) =
	Q\left|_{x_0 \leftarrow \frac{P|_{x_i = 0}}{x_{n+1}}} \right. \in\PP[x_{n+1}^{\pm 1}]$.

	\item If $d$ is the factor of $G$ shared with $P\big|_{x_i = 0}$ i.e., $d = \gcd(G, (P|_{x_i = 0})^k)$ in $\Z[x_1, \ldots, \widehat{x}_i, \ldots, x_{n-1}]$ for some sufficiently large $k\in\N$, then let $H = G/d$.
	\item Finally, define $\tau_P(Q)$ to be the downshift of $MH$, where $M \in \mathcal{L}(x_1, \ldots, \widehat{x}_i, \ldots, x_{n})$
	is such that $MH \in\Z[x_1, \ldots, \widehat{x}_i, \ldots, x_{n}]$ and $MH$ is not divisible by any $x_j$.
\end{enumerate}

\begin{rem}
If $Q$ does not depend on $x_0$, then $H = G = Q$, $M = 1$ and so $\tau_P(Q)$ is simply the downshift of $Q$.
\end{rem}

It is not immediately clear that $\tau$ maps polynomials from $\PP_{i}$ to $\PP_{i-1}$ for all $i$.
The two propositions below show that this is the case.

\begin{prop}\label{nodepend}
If $Q\in\PP_{i}$, then $R = \tau_{P}(Q)$ does not depend on $x_{i-1}$.
\end{prop}

\begin{proof}
If $P_i$ does not depend on $x_1$, the statement follows because $P_{i-1} = \tau(P_i)$ is the downshift of $P_i$.
If $P_i$ depends on $x_1$, then in the computation of $\tau(P_i)$, we define $G$ as
the Laurent polynomial resulting from replacing $x_0$ in $Q$ by an expression wherein we made the substitution $x_i = 0$.
In particular, $G$ does not depend on $x_i$. Then, $H$ does not contain $x_i$ and neither does $M$, by definition.
Hence, $P_{i-1} = \tau(P_i)$, which is the downshift of $MH$, does not contain $x_{i-1}$.
\end{proof}

\begin{prop}\label{irreducible}
If $Q\in\PP_i$, then $R = \tau_P(Q)$ is irreducible in $\PP$ and is not divisible by any of the $x_j$.
\end{prop}

\begin{proof}
From the definition of $\tau$, it is clear that $R = \tau(Q)$ is not divisible by any $x_j$. It then suffices to show $R$ is irreducible.
This is clear if $Q$ does not depend on $x_0$, so assume $Q$ depends on $x_0$.
Write 
\[ Q = \sum_k{f_k x_0^k}, \]
where $f_k\in\Z[x_1, \ldots, \widehat{x}_i, \ldots, x_{n-1}]$ for all $k$.
Then,
\[ G = Q\big|_{x_0\leftarrow \frac{P|_{x_i = 0}}{x_{n}}} =
\sum_k f_k \left(\frac{P|_{x_i = 0}}{x_{n}}\right)^k,\]
and $H$ is $G$ divided by all common factors it shares with $P_0|_{x_i\leftarrow 0}$. Finally, $R$ is the downshift of $MH$ for some Laurent monomial $M$.
As $M$ is a unit in $\PP$, it will suffice to show $H$ is irreducible.
Let $d$ be a nonunit factor of $H$. From the definition of $H$, $d$ is not a factor of $P|_{x_i = 0}$.

If $d$ is independent of $x_{n}$ then $d \mid f_k$ for all $k$, which implies $d \mid Q$, contradicting the irreducibility of $Q$.

If $d$ depends on $x_{n}$, write $d = d(x_{n})$, so,
\[ d\left(\frac{P_0|_{x_i\leftarrow 0}}{x_{n}}\right) \text{ divides } G\left(\frac{P_0|_{x_i\leftarrow 0}}{x_{n}}\right)=Q(x_{n}). \]
This again contradicts that $Q$ is irreducible.
\end{proof}

\vspace{.2in}

Given an irreducible polynomial $P = P(x_1, \ldots, x_{n-1})\in\PP$, we generate a seed $\bf (x, P)$ by letting
$P_0 = P$, $P_{n-1}$ be the downshift of $P$, and recursively defining $P_i = \tau_P(P_{i+1})$ for $i = n-1, n-2, \ldots, 1$.
From Propositions \ref{nodepend} and \ref{irreducible}, it is clear that $\bf (x, P)$ is a valid seed.
For example, for $n = 3$, the polynomial $P = x_1x_2 + x_3^2$ generates the seed
$\{(x_0, x_1x_2 + x_3^2), (x_1, x_2^3 + x_0x_3^2), (x_2, x_0^2 + x_1x_3), (x_3, x_0x_1 + x_2^2)\}$.
The following proposition gives a sufficient condition for asserting that $\bf (x, P)$ is a period 1 seed.

\begin{prop}\label{period1}
Let $\widehat{P}_0$ be the exchange Laurent polynomial of $P_0$ for the generated seed $\bf (x, P)$.
If $P_0 = \tau_P(P_1)$ and $\widehat{P}_0 = P_0 = P(x_1, \ldots, x_{n-1})$, then $\bf (x, P)$ is a period 1 seed.
In particular, $P$ generates a Laurent phenomenon sequence.
\end{prop}

\begin{proof}
We remarked above that $\bf (x, P)$ is a valid seed. It is also clear that $P_{n-1}$ is the downshift of $P_0$.
Finally, observe that if $\widehat{P}_0 = P_0$, then the definitions of $\tau_P$ and $\mu_0$ coincide.
Therefore, the seed $\bf (x, P)$ is a period 1 seed, as desired.
\end{proof}

If $t = \bf(x, P)$ is a seed generated by $P\in\PP_0$ and is such that $P_0 = \tau_P(P_1)$, we say $t$ has {\it pseudoperiod} $1$. Proposition \ref{period1} can then be rephrased as saying that if $\widehat{P}_0 = P_0$, then $t$ has period 1.
The following conjecture, in conjunction with Proposition \ref{period1}, would show that period and pseudoperiod are equivalent definitions in this context.

\begin{conj}\label{hatconj}
Let $\widehat{P}_0$ be the exchange Laurent polynomial of $P_0$ for the generated seed $\bf (x, P)$. Then $\widehat{P}_0 = P_0$.
\end{conj}

In the rest of this paper, we classify certain families of polynomials that generate pseudoperiod 1 seeds.
In addition, we find many examples of pseudoperiod 1 seeds.
In all cases, we can show that the seeds are, indeed, period 1 seeds, using the Lemma below and Proposition \ref{period1}.

For simplicity of terminology, in cases where the Lemma below is satisfied, we will simply say period 1 instead of pseudoperiod 1.

\begin{lem}\label{hat}
$\widehat{P}_0 = P_0$ if either of the two conditions holds
\begin{enumerate}
	\item $P_j$ depends on $x_0$ whenever $P_0$ depends on $x_j$.
	\item All polynomials $P_j$, $0\leq j\leq n-1$, have the same number $d$ of terms.
\end{enumerate}
\end{lem}

\begin{proof}
From the construction of the $P_j$, it is clear that
$P_0\big|_{x_j\leftarrow P_j/x} \in\Z[x_0^{\pm 1}, \ldots, x_{j-1}^{\pm 1}, x^{\pm 1}, x_{j+1}^{\pm 1}, \ldots, x_{n-1}^{\pm 1}]$.
It then suffices to show that $P_0\big|_{x_j\leftarrow P_j/x}$ is not divisible by $P_j$, or equivalently
that $P_0\big|_{x_j = 0}$ is not divisible by $P_j$.

(1) If $P_0$ depends on $x_j$, then $P_j$ depends on $x_0$ by assumption.
From Proposition \ref{nodepend}, $P_0$ does not depend on $x_0$ and therefore neither does $P_0\big|_{x_j = 0}$.
Then $P_j$ cannot divide $P_0\big|_{x_j = 0}$.

If $P_0$ does not depend on $x_j$, then $P_0\big|_{x_j = 0} = P_0$.
As both $P_0$ and $P_j$ are irreducible and not divisible by any $x_k$, we only need that $P_j \neq P_0$ for $j > 0$.
Let $m, M$ be the minimum and maximum indices $i$ such that $P_0$ depends on $x_i$.
We claim that $P_j$ either does not depend on $x_M$ or it depends on some $x_k$ with $k < m$; this immediately implies $P_j \neq P_0$ for $j > 0$.
If there is no intermediate polynomial $P_{j'}$ with $j' > j$ that depends on $x_0$, then $P_s$ is the downshift of $P_{s+1}$ for all $s \geq j$. Since the maximum index upon which $P_{n-1} = P(x_0, \ldots, x_{n-2})$ depends is $M - 1$, then the maximum index upon which $P_j$ depends is also smaller than $M$; in particular $P_j$ does not depend on $x_M$.
If there is some intermediate polynomial $P_{j'}$ with $j' > j$ that depends on $x_0$, let $j_0$ be the smallest such index $j_0 > j$ (so $P_s$ is the downshift of $P_{s+1}$ for all $j\leq s < j_0$).
Recall tht the polynomial $P_{j_0 - 1}$ comes from
\[ P_{j_0}\left|_{x_0 \leftarrow \frac{P|_{x_{j_0} = 0}}{x_{n}}} \right. . \]
Hence, $P_{j_0 - 1}$ depends on $x_{m-1}$ unless $P_{j_0}\big|_{x_0 = 0}$ is divisible by $P = P_0$.
Since $P_{j_0}$ depends on $x_0$, $P_0$ depends on $x_{j_0}$. From Proposition \ref{nodepend}, $P_{j_0}$ does not depend on $x_{j_0}$; therefore $P_0$ cannot divide $P_{j_0}|_{x_0 = 0}$.
Therefore $P_{j_0-1}$ depends on $x_{m-1}$. The polynomial $P_j$, which is the result of $j_0 - j - 1$ downshifts from $P_{j_0 - 1}$, then depends on $x_{m'}$ for some $m' < m$.

(2) If $P_0$ depends on $x_j$, then $P_0\big|_{x_j = 0}$ has at most $d-1$ terms. Thus $P_j$, which has $d$ terms, cannot divide it.

If $P_0$ does not depend on $x_j$, then $P_0|_{x_j\leftarrow P_j/x} = P_0$. As both $P_0$ and $P_j$ are irreducible and not divisible by any $x_k$, we only need to show $P_j \neq P_0$ for $j > 0$.
The argument is the same as in part (1) except for the reason why $P_0$ does not divide $P_{j_0}|_{x_0 = 0}$.
In this case, it is because $P_{j_0}|_{x_0 = 0}$ has at most $d-1$ terms (as $j_0$ was defined as an index for which $P_{j_0}$ depends on $x_0$) and $P_0$ has $d$ terms.
\end{proof}

\begin{rem}
One can see that if $P$ generates some period 1 seed, then such seed must be the one described in Proposition \ref{period1}. If we begin with $P \in \PP$ and follow the process mentioned above (recursively obtain the intermediate polynomials $P_j$, $0<j<n-1$), we may have that one of the conditions in Proposition \ref{period1} is not satisfied; in that case, $P$ is not a period 1 polynomial.
\end{rem}

We next use \cite[Proposition 2.10]{lp}, which says that if $\bf (x', F') = \mu_i(x, F)$, then $\bf (x, F) = \mu_i(x', F')$, to devise an analogue of $\tau$ that instead takes polynomials from $\PP_{i-1}$ to polynomials in $\PP_i$.

We define the mapping $\kappa = \kappa_P$, that is the inverse of $\tau$ as follows.

Given a polynomial $P\in\PP_0 = \Z[x_1, \ldots, x_{n-1}]$, let $P' = P(x_0, x_1, \ldots, x_{n-2})\in\PP_{n-1}$ and
$\kappa: \PP \to \PP$ a map which takes polynomials from $\PP_i$ to polynomials in $\PP_{i+1}$ for all $i\geq 0$.
If $Q\in\PP_i$, then $\kappa_P(Q)\in\PP_{i+1}$ is computed according to the following rules.

\begin{enumerate}
	\item Let $\displaystyle G' = G'(x_{-1}, x_1, \ldots, \widehat{x}_i, \ldots, x_{n-1}) = Q\big|_{x_{n-1} \leftarrow \frac{P'|_{x_i = 0}}{x_{-1}}}\in\PP[x_{-1}^{\pm 1}]$.
    \item If $d'$ be the factor of $G'$ shared with $P'\big|_{x_i = 0}$, i.e., $d' = \gcd(G', (P'|_{x_i = 0})^k)$ in $\Z[x_1, \ldots, \widehat{x}_i, \ldots, x_{n-1}]$ for some sufficiently large $k\in\N$, then let $H' = G'/d'$. 
	\item Finally, let $R = \kappa_P(Q)$ be the upshift of $M'H'$, where $M' \in \mathcal{L}(x_{-1}, x_1, \ldots, \widehat{x}_i, \ldots, x_{n-1})$ is such that $M'H' \in\Z[x_{-1}, x_1, \ldots, \widehat{x}_i, \ldots, x_{n-1}]$ and is not divisible by any $x_j$.
\end{enumerate}

\begin{rem}
If $Q$ does not depend on $x_{n-1}$, then $H' = G' = Q$, $M' = 1$ and so $\kappa_P(Q)$ is simply the upshift of $Q$.
\end{rem}

The proof that $\kappa$ is a well defined map comes from the analogous statements of
Propositions \ref{nodepend} and \ref{irreducible} to $\kappa$.
Given an irreducible polynomial $P$, choose $0 < k < n-1$. We generate a seed $\bf (x, P)$ by letting
$P_0 = P$, $P_{n-1} = P'$ be the downshift of $P$ and recursively defining $P_i = \tau_P(P_{i+1})$ for all $k < i < n$
and $P_i = \kappa_P(P_{i-1})$ for all $0 < i < k$. A refinement of Proposition \ref{period1} is then

\begin{prop}\label{refinement}
Let $\widehat{P}_0$ be the exchange Laurent polynomial of $P_0$ for the generated seed $\bf (x, P)$.
If $P_k = \kappa(P_{k-1})$, or equivalently $P_{k-1} = \tau(P_k)$, and $\widehat{P}_0 = P_0$, then
$\bf (x, P)$ is a period 1 seed.
\end{prop}

\begin{rem}
We have implemented the above algorithm (with $k=\lfloor n/2\rfloor$) in Sage at \url{http://sage.lacim.uqam.ca/home/pub/23/}. This can be used to test whether a given polynomial $P$ is period 1.
\end{rem}

\section{Statements of results and conjectures}\label{results}

In this section, we present our main results. Their proofs will be presented in the remaining sections. In the first subsection, we give our classification theorems, while in the second subsection, we give a proposition asserting that several large families of polynomials are 1 periodic.

\subsection{Classification theorems}

We first classify all period 1 polynomials when $n=2,3$.

\begin{thm}\label{n2thm}
For $n = 2$, the only period 1 polynomials $P$ are
\begin{enumerate}
    \item Irreducible polynomials that are monic and palindromic, i.e., that satisfy $x^{\deg(P)}\cdot P(\frac{1}{x}) = P(x)$.
    \item Irreducible polynomials of even degree that are monic and antipalindromic, i.e., that satisfy $x^{\deg(P)}\cdot P(\frac{1}{x}) = -P(x)$.
	\item Monic irreducible polynomials of degree $2$.
\end{enumerate}
\end{thm}

\begin{rem}
Gregg Musiker showed in \cite{musiker} that the only polynomials $P$ that generate Laurent phenomenon sequences are the ones in the above theorem. Thus Theorem \ref{n2thm} shows that when $n=2$, period 1 polynomials are exactly the polynomials that generate Laurent phenomenon sequences.
\end{rem}

\begin{thm}\label{n3thm}
For $n = 3$, the only period 1 polynomials $P$ are
\begin{enumerate}
    \item $P = x_1x_2 + ax_1 + ax_2$, for any $a\in\Z, a\neq 0$,
    \item $P = x_1x_2 + ax_1 - ax_2$, for any $a\in\Z, a\neq 0$,
    \item $P = x_1 - x_2 - 1$,
    \item $P = -x_1 + x_2 - 1$,
    \item $P = x_1x_2 + ax_1 + ax_2 + b$, for any $a, b\in\Z$, not both of which are $0$,
    \item $P = x_1^2 + x_2^2 + ax_1x_2 + bx_1 + bx_2 + c$, for any $a, b, c\in\Z$,
    \item $P = -x_1^2 -x_2^2 + ax_1x_2+b$, for any $a,b\in\Z$,
    \item $P = \pm x_1x_2+a$, for any $a\in\Z, a\neq 0$,
    \item $P = 1+x_1^mx_2^n+\sum_{0<i<m\atop 0<j<n}C_{i,j}(x_1^ix_2^j+x_1^{m-i}x_2^{n-j})$, for any $C_{i, j} \in \Z$, $m, n \in\N_{>0}$,
    \item $P = -1+(-1)^{m+1}x_1^mx_2^n+\sum_{0<i<m\atop 0<j<n}C_{i,j}(x_1^ix_2^j+(-1)^{m+j+i}x_1^{m-i}x_2^{n-j})$ for any $C_{i, j}\in\Z$, $m, n\in\N_{>0}$, $m\equiv n\mod 2$.
\end{enumerate}
\end{thm}

\begin{rem}
The arbitrary coefficients and exponents in Theorem \ref{n3thm} must be such that $P$ is irreducible and not divisible by any $x_j$.
\end{rem}

Our final classification theorem comes from our own definition of \emph{Double Quivers}. The family of polynomials we found includes those that are classified by the main theorem in \cite{fm}.

\begin{thm}\label{binomialthm}
The binomial $P$ generates a period 1 seed which corresponds to a double quiver if and only if it is of the form \[ P = \prod_{1\leq i\leq n}{x_i^{a_i}} + \prod_{1\leq i\leq n}{x_i^{b_i}}, \] where $a_i, b_i\in\Z_{\geq 0}$ are such that $a_i = 0 \Longleftrightarrow a_{n-i} = 0$ and $b_j = 0 \Longleftrightarrow b_{n-j} = 0$.
\end{thm}

Finally, the families of polynomials we have found, and that we present in the next subsection, give rise to the following conjectures:

\begin{conj}\label{multilinearconj}
If $P$ is a multilinear polynomial with positive coefficients that generates a period 1 seed, then $P(x_1, x_2, \ldots, x_n) = P(x_n, x_{n-1}, \ldots, x_1)$.
\end{conj}

\begin{conj}\label{linearconj}
If $n$ is odd, no linear polynomial with positive coefficients generates a period 1 seed.
If $n$ is even, the only linear polynomial $P$ with positive coefficients that generates a period 1 seed is $x_{n/2} + 1$.
\end{conj}

\begin{conj}\label{symconj}
The only symmetric polynomials $P$ with positive coefficients that generate period 1 seeds are either of the form \[ P = \sum_{i=1}^{n-1}{x_i^2} + M(x_1, \ldots, x_{n-1}), \] where $M$ is any multilinear symmetric polynomial, or of the form \[ P = \sum_{1\leq i < j \leq n-1}{x_ix_j} + A\sum_{i=1}^{n-1}{x_i} + B, \] for odd $n$.
\end{conj}

\subsection{Families of period 1 polynomials}

\begin{thm} \label{families}
The following families of polynomials $P$ are 1 periodic.

\begin{enumerate}
    \item {\it Symmetric with second powers polynomial.}
\begin{align*}
P = S + A_1E_1 + \ldots A_{n-1}E_{n-1} + A,
\end{align*}
for any coefficients $A, A_1, \ldots, A_{n-1}\in\Z$, where $E_k = \sum_{1\leq i_1 < \ldots < i_k\leq n-1}{x_{i_1}\ldots x_{i_k}}$ for all $1\leq k \leq n$ and $S = \sum_{i=1}^{n-1}{x_i^2}$.

For example, $P = x_1^2 + x_2^2 + 2x_1x_2 + 5$ when $n = 3$.
    \item {\it Sink-type binomial.}
\begin{align*}
P = x_1^{a_1}x_2^{a_2}\ldots x_{n-1}^{a_{n-1}}+1,
\end{align*}
where $a_i=0 \Longleftrightarrow a_{n-i}=0$ for all $i$. For example, $P = x_1^2x_3^3x_5 + 1$ when $n = 6$.
    \item {\it Extreme polynomial.} 
\begin{align*}
P = x_1x_{n-1} + A\cdot\sum_{i=1}^{n-1}{x_i} + B,
\end{align*}
for any coefficients $A, B\in\Z$. For example, $P = x_1x_3 + 3(x_1 + x_2 + x_3) + 2$ when $n = 4$.
    \item {\it Singleton polynomial.} If $n\in\N$ is even, let $P$ is a single variable monic irreducible polynomial that is palindromic ($x^{\deg(P)}\cdot P(1/x) = P(x)$), or antipalindromic ($x^{\deg(P)}\cdot P(1/x) = - P(x)$), or $P = x_{n/2}^2 + Ax_{n/2} + B$ for any $A, B\in\Z$. For example, $P = x_2^2 + 2x_2 - 7$ when $n=3$.
    \item {\it Chain polynomial.} If $n\in\N, n>2$ is odd,
\begin{align*}
P = \sum_{i=1}^{n-2}{x_ix_{i+1}} + A\cdot\sum_{i=1}^{n-1}{x_i} + B,
\end{align*}
for any coefficients $A, B\in\Z$. For example, $P = x_1x_2 + x_2x_3 + x_3x_4 + 2(x_1+x_2+x_3+x_4) + 3$ when $n = 5$.
    \item {\it Multilinear symmetric polynomial.} If $n\in\N, n > 2$, 
\begin{align*}
P = E_2 + A\cdot E_1 + B,
\end{align*}
for any coefficients $A, B\in\Z$, where the $E_i$ are the elementary symmetric polynomials. For example,  $P = x_1x_2 + x_2x_3 + x_1x_4 + x_2x_3 + x_2x_4 + x_3x_4 - 3(x_1+x_2+x_3+x_4) + 1$ when $n = 5$.
    \item {\it $r$-Jumping polynomial.} If $r, n\in\N$ are such that $n\geq 2r+1$ and $n \equiv 1 \pmod{r}$,
\begin{align*}
P = \sum_{i = 0}^{\frac{n-1}{r} - 1}{x_{ri + 1}\cdot x_{ri + r}} + A,
\end{align*}
for any $A\in\Z$. For example, $P = x_1x_3 + x_4x_6$ when $n = 7$.
    \item {\it $r$-Hopping polynomial.} If $r, n\in\N$ be such that $n\geq 2r + 2$ and $n \equiv 1 \pmod{r}$,
\begin{align*}
P = \sum_{i = 0}^{\frac{n-1}{r} - 1}{x_{ri + 1}\cdot x_{ri + r}} + A\cdot\sum_{i=0}^{\frac{n-1}{r}-2}{x_{ri+r}\cdot x_{ri+r+1}} + B,
\end{align*}
for any $A, B\in\Z$. For example, $P = x_1x_3 - 2x_3x_4 + x_4x_6 + 3$ when $n = 7$.

{\bf Note:} The $r$-Jumping polynomials are special cases of the $r$-Hopping polynomials (when $A = 0$). We distinguish them because we found a conserved quantity for sequences generated by $r$-Jumping polynomials, but not by $r$-Hopping polynomials (see Section \ref{conserved}).
    \item {\it Flip-symmetric binomial.} If $L, R \subset [n-1]$ are disjoint subsets such that $i\in L \Longleftrightarrow n-i\in L$ and $i\in R \Longleftrightarrow n-i\in R$, and if $a: L\cup R\rightarrow\N$ is any map into the positive integers, then,
\begin{align*}
P = \prod_{i\in L}{x_i^{a(i)}} + \prod_{i\in R}{x_i^{a(i)}}.
\end{align*}
For example, $P = x_1^3x_7^2 + x_4^3x_2x_6$ when $n = 8$.

{\bf Note:} The Somos-4 and Somos-5 polynomials ($x_1x_3 + x_2^2$ and $x_1x_4 + x_2x_3$) are particular cases of flip-symmetric binomials. The family (2) of sink-type polynomials are also particular cases of flip-symmetric polynomials (when $R = \emptyset$).
    \item {\it Balanced polynomial.} If $L, R \subset [n-1]$ are disjoint subsets such that $i\in L \Longleftrightarrow n-i\in L$ and $i\in R \Longleftrightarrow n-i\in R$, and $a: L\cup R\rightarrow\N$ is any map into the positive integers. Then for any $m>1$, write $\displaystyle M_1=\prod_{i\in L}{x_i^{a(i)}}, M_2=\prod_{i\in R}{x_i^{b(i)}}$ and,
\begin{align*}
P = M_1^m + M_2^m + \sum_{i=1}^{\lfloor \frac{m}{2}\rfloor} {A_i\cdot\left(M_1^{i}M_2^{m-i} + M_1^{m-i}M_2^{i}\right)},
\end{align*}
for arbitrary coefficients $A_i\in\Z$, $1\leq i\leq \lfloor\frac{m}{2}\rfloor$. For example, $P = x_3^{12}x_6^4 + x_2^4x_7^8 + 2(x_2x_3^9x_6^3x_7^2 + x_2^3x_3^3x_6x_7^6)+3x_2^2x_3^6x_6^2x_7^4$ when $n = 9$.
     \item {\it Vector sum polynomial.} For $a_1, \ldots, a_{n-1}\in\N$ and a finite set $B$ of vectors $(b_1, \ldots, b_{n-1})\in\N^{n-1}$ such that $0 < b_i < a_i$ for all $i$, then,
\begin{align*}
P = 1 + x_1^{a_1}\ldots x_{n-1}^{a_{n-1}} + \sum_{b\in B}{(C_b\cdot x_1^{b_1}\ldots x_{n-1}^{b_{n-1}} + C_b\cdot x_1^{a_1 - b_1}\ldots x_{n-1}^{a_{n-1}-b_{n-1}})},
\end{align*}
for arbitrary coefficients $C_b\in\N$. For example, $P = 1 + x_1^3x_2^2x_3^4x_4^2 + 2x_1x_2x_3^2x_4 + 2x_1^2x_2x_3^2x_4$ when $n = 5$.
    \item {\it Little Pi polynomial.} For $k, n\in\N$ such that $n > 2k$ and $n \neq 3k$, then
\begin{align*}    
P = Ax_k + Ax_{n-k} + x_{2k}x_{n - 2k},
\end{align*}
for any $A\in\Z$. For example, $P = 2x_2 + 2x_5 + x_4x_3$ when $n = 7$, $k = 2$.
    \item {\it Pi polynomial.} For $k, n, a_1, b_1, a_2, b_2\in\N$ such that $n > 2k$, $n \neq 3k$ and $a_1+b_1=a_2+b_2$, then
\begin{align*}    
P = Ax_k^{a_1}x_{b_1}^b + Bx_k^{a_2}x_{n-k}^{b_2} + x_{2k}x_{n - 2k}
\end{align*}
for any $A, B\in\Z$. For example, $P = -2x_2^1x_6^4 + 3x_2^2x_6^3 + x_4^2$ when $n = 8$, $k = 2$.
\end{enumerate}
\end{thm}

\begin{rem}
In each case of the theorem above, we omitted saying that the coefficients and exponents are such that $P$ is irreducible and not divisible by any $x_j$.
The following important corollary will also follow easily from the proof of Theorem \ref{families} and Theorem \ref{hat}.
\end{rem}
\begin{rem}
Recently, Hone and Ward found independently the Laurent phenomenon for extreme polynomials (family (3) in Theorem \ref{families}). They do a thorough study of this family of polynomials in \cite{hw}.
\end{rem}

\begin{cor}\label{laurentproperty}
All polynomials $P$ from Theorem \ref{families} generate Laurent phenomenon sequences.
\end{cor}

\begin{conj}
Let $k, n, a_1, b_1, a_2, b_2\in\N$ be such that $n > 2k$, $n\neq 3k$ and $a_1+b_1=a_2+b_2$. Consider the polynomial
\begin{align*}    
P = (Ax_k^{a_1}x_{n-k}^{b_1} + Bx_k^{a_2}x_{n-k}^{b_2})\cdot M + x_{2k}x_{n - 2k},
\end{align*}
for any $A, B\in\Z$ and monomial $\displaystyle M = \prod_{i=1\atop i\neq 2k,n-2k}^{n-1}{x_i^{c_i}}$, where $c_i = 0 \Longleftrightarrow c_{n-i} = 0$ for all $i$.

Then $P$ is a period 1 polynomial and generates a Laurent phenomenon sequence.
\end{conj}

We also will prove the following lemmas that can be applied to known period 1 polynomials to yield new ones:

\begin{lem}\label{explemma} {\bf (Expansion Lemma)}
If $F = F(x_1,x_2,\ldots, x_{n-1})$ generates a period 1 seed,
then for any $k \in \N$, so does the polynomial $G = G(x_1,x_2,\ldots,x_{nk-1}) = F(x_{k},x_{2k},\ldots,x_{(n-1)k})$. We call $G$ the $k$-expansion of $F$.
\end{lem}

\begin{lem}\label{reflemma} {\bf (Reflection Lemma)}
If $F = F(x_1,x_2,\ldots, x_{n-1})$ generates a period 1 seed,
then so does $G = G(x_1,x_2,\ldots,x_{n-1}) = F(x_{n-1},x_{n-2},\ldots,x_1)$.
\end{lem}

\begin{rem}
Observe that the reflection lemma, applied to the families of polynomials in Theorem \ref{families}, always gives another member of the same family.
\end{rem}

\section{Polynomials arising from double quivers}\label{doublequiversec}

\subsection{Binomial Seeds and Double Quivers}
In this section we find all period 1 binomials with a mild mutuality condition. To do this, we first introduce a new representation for binomial seeds, which we call a double quiver. The main constraint of a normal quiver that our double quiver removes is that binomial seeds represented by a quiver have to be mutual, i.e., if $x_i$ appears in $P_j$, then $x_j$ appears in $P_i$ with the same degree.
\begin{df}
A {\em double quiver} $Q$ is a finite set of vertices with directed half-edges between vertices. Between each pair of vertices $i$ and $j$, there can be edges between them attached at $i$, as well as edges between them attached at $j$. We allow multiple half-edges at each vertex, but not $2$-cycles, i.e., there cannot be edges from $i$ to $j$ as well as edges from $j$ to $i$ all attached at $i$. We also do not allow self-loops.

The {\em B-matrix} $B = (b_{i,j})_{n\times n}$ of a double quiver with $n$ vertices is defined as follows. The magnitude $|b_{i,j}|$ is the number of half-edges between vertex $i$ and vertex $j$ that are attached at $i$. If the edges are outgoing from vertex $i$, then $b_{i,j} > 0$; if the edges are incoming to $i$, then $b_{i,j} < 0$. Conversely, each $n\times n$ integer matrix with $0$'s in its diagonal corresponds to a double quiver. For convenience, we will index the rows and columns of $B$ from $0$ to $n-1$. The (LP algebra) seed corresponding to a B-matrix $B$ is $\bf (x, P)$, where $\mathbf{x} = \{x_0, \ldots, x_{n-1}\}$ and the intermediate polynomials are, for all $i$:
\[ P_i = \prod_{j : b_{i, j} > 0}{x_j^{b_{i, j}}} + \prod_{j: b_{i, j} < 0}{x_j^{-b_{i, j}}}. \]
\end{df}

\begin{exam}
Figure \ref{DQ1} shows a double quiver with $3$ vertices. There is a half-edge from $x_1$ to $x_0$ attached at $x_0$, a half-edge from $x_1$ to $x_0$ attached at $x_1$, two half-edges from $x_0$ to $x_2$ attached at $x_0$, a half-edge from $x_0$ to $x_2$ attached at $x_2$, three half-edges from $x_2$ to $x_1$ attached at $x_1$ and no half-edges from $x_1$ to $x_2$ attached at $x_2$.
\end{exam}

\begin{df}
A vertex $i$ of a double quiver is {\em{mutable}} if, whenever there are half-edges between $i$ and $j$ attached at $j$, then there are also half-edges between $i$ and $j$ attached at $i$. In terms of the B-matrix, vertex $i$ is mutable if for all other vertices $j$,
$b_{j,i} \neq 0$ implies $b_{i,j} \neq 0$.
\end{df}

\begin{exam}
In the double quiver of figure \ref{DQ1}, $x_0$ and $x_1$ are mutable, but $x_2$ is not mutable since there are half-edges from $x_2$ to $x_1$ attached at $x_1$, but no half-edges between $x_1$ and $x_2$ attached at $x_2$.
\end{exam}
\setlength{\unitlength}{0.5cm}
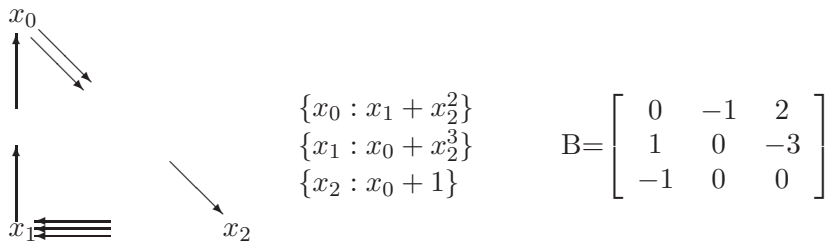
\begin{figure}[H]
\begin{center}
\begin{picture}(20,6)
\linethickness{0.5pt}

\put(-0.7,-0.2){$x_1$}
\put(-0.7,5.5){$x_0$}
\put(5,-0.2){$x_2$}

\put(2,0){\vector(-1,0){2}}
\put(2,0.2){\vector(-1,0){2}}
\put(2,-0.2){\vector(-1,0){2}}
\multiput(-0.5,0.2)(0,3){2}{\vector(0,1){2}}
\put(0.1,5.3){\vector(1,-1){1.4}}
\put(3.6,1.8){\vector(1,-1){1.4}}
\put(-0.1,5.1){\vector(1,-1){1.4}}

\put(7,3){$\{x_0: x_1+x_2^2\}$}
\put(7,2){$\{x_1: x_0+x_2^3\}$}
\put(7,1){$\{x_2: x_0+1\}$}

\put(14,2){B=$\left[\begin{array}{ccc}
  0 & -1 & 2\\
  1 & 0  & -3\\
  -1 & 0 & 0\end{array}\right]$}
\end{picture}
\end{center}
\caption{Example of a Double Quiver}\label{DQ1}
\end{figure}

\begin{df}
We define {\em mutation} at a mutable vertex $k$ of a double quiver $Q$ with vertices $\{0, 1, \ldots ,n-1\}$ to be the application of the map $\tau_k$ that takes $Q$ to a new double quiver $\tau_k(Q)$ via the following steps:
\begin{enumerate}
\item Add a half-edge $i \rightarrow j$ attached at $i$, for each pair of half-edges $i\rightarrow k$ attached at $i$ and $k\rightarrow j$ attached at $k$. Also add a half-edge $j \rightarrow i$ attached at $i$, for each pair of half-edges $j \rightarrow k$ attached at $k$ and $k\rightarrow i$ attached at $i$.
\item Reverse the direction of half-edges between vertex $k$ and node $i$, for all $i\neq k$.
\item Successively pick $2$-cycles and remove both half-edges until no $2$-cycles remain.
\end{enumerate}

The mutation of a double quiver corresponds to the mutation of the corresponding LP algebra seed. Let $\bf (x, P)$ be the LP algebra seed associated to the double quiver $Q$ and $\bf (x', P')$ the LP algebra seed associated to $Q'$, the double quiver resulting from mutating $Q$ at $k$. Then the intermediate polynomials $P_j'$ are the intermediate polynomials of the seed $\mu_k(\mathbf{x, P})$, where $\mu_k$ is seed mutation as defined in Section 2. We will be able to find all period 1 binomials $P$ that satisfy some mild conditions regarding their corresponding double quiver $Q$. Observe that a period 1 seed whose exchange polynomials are all binomials has a corresponding period 1 double quiver $Q$. However, it will be easier to work with period 1 B-matrices; next, we give the corresponding definition of mutation for B-matrices.

Denote by $\mathbf{1}_S$ the indicator variable of $S$.
Mutation at vertex $k$ corresponds to a {\it mutation of the B-matrix} of the double quiver that maps it to $\tau_k(B) = \tilde{B} = (\tilde{b}_{i,j})_{n\times n}$, such that
\begin{align}\label{bmatrixmutation}
\tilde{b}_{i,j}= \left\{ \begin{array}{lc}
-b_{i,j}&i=k \text{ or } j=k\\
b_{i,j}+b_{i,k} \cdot |b_{k,j}| \cdot \textbf{1}_{\{b_{k,i}b_{k,j}<0\}}& \text{otherwise}\end{array} \right.
\end{align}
\end{df}
\begin{exam}
If we mutate the double quiver in Figure (\ref{DQ1}) at $x_0$, we obtain the following double quiver

\setlength{\unitlength}{0.5cm}
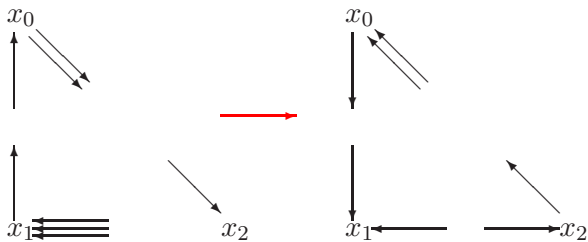
\begin{figure}[H]
\begin{center}
\begin{picture}(10,6)
\linethickness{0.5pt}

\put(-0.7,-0.2){$x_1$}
\put(-0.7,5.5){$x_0$}
\put(5,-0.2){$x_2$}

\put(2,0){\vector(-1,0){2}}
\put(2,0.2){\vector(-1,0){2}}
\put(2,-0.2){\vector(-1,0){2}}
\multiput(-0.5,0.2)(0,3){2}{\vector(0,1){2}}
\put(0.1,5.3){\vector(1,-1){1.4}}
\put(3.6,1.8){\vector(1,-1){1.4}}
\put(-0.1,5.1){\vector(1,-1){1.4}}

{\color{red}
\put(5,3){\vector(2,0){2}}}

\put(8.3,-0.2){$x_1$}
\put(8.3,5.5){$x_0$}
\put(14,-0.2){$x_2$}

\put(11,0){\vector(-1,0){2}}
\put(12,0){\vector(1,0){2}}
\multiput(8.5,2.2)(0,3){2}{\vector(0,-1){2}}
\put(10.5,3.9){\vector(-1,1){1.4}}
\put(14,0.4){\vector(-1,1){1.4}}
\put(10.3,3.7){\vector(-1,1){1.4}}

\end{picture}
\end{center}
\caption{Double Quiver Mutation at $x_0$}\label{DQ2}
\end{figure}

\end{exam}

\begin{rem}\label{quivervsdouble}
Double quivers are generalizations of (normal) quivers in the following sense:
\begin{enumerate}
\item A quiver Q can be regarded as an example of a double quiver. Split each edge $i \rightarrow j$ into two half-edges. Then attach one of them to $i$ and the other to $j$. The mutation rules for double quivers and for quivers agree with each other.
\item The cluster algebra $\mathcal{A}$ defined by any skew-symmetrizable matrix $B$ can be realized as a double quiver. In fact, $B$ is associated to a double quiver $Q$ and to a seed $t$ that gives rise to a LP algebra $\mathcal{A}(t)$ that is identical to $\mathcal{A}$ and the mutation rules agree. Furthermore, if $v$ is a vertex in the double quiver $\tilde{Q}$, that is the result of mutating $Q$ at $v$, then $v$ is mutable in $\tilde{Q}$.
\item Fomin and Zelevinsky defined cluster algebras in their foundamental paper \cite{fz2} by sign-skew-symmetric matrices. In this definition, it was required that any sequence of mutations yields another sign-skew-symmetric matrix. Our double quivers can be regarded as a direct generalization of cluster algebras defined by sign-skew-symmetric matrices. For one thing, we do not require the matrix $B$ to be sign-skew-symmetric. For another, we have fewer restrictions on the mutation rules; we define mutability at a vertex, so that double quivers where some mutation sequences are invalid but others are not can still be considered.  
\end{enumerate}
\end{rem}

\subsection{1 Periodicity}
In this section, we examine more precisely the notion of a period 1 double quiver. We also prove a weaker version of Theorem \ref{binomialthm}.

Let $Q$ be a double quiver and $B$ be the matrix (not necessarily skew-symmetric) determined by $Q$. We say that {\it $Q$ has period 1} if mutating at $0$ and relabeling the vertices $(0,1,2,\ldots n-1) \rightarrow (n-1,0,1,\ldots n-2)$ gives back the original double quiver $Q$. In particular, if $Q$ has period 1, then its vertex $0$ is mutable, meaning in terms of B-matrices that $b_{k,0}\neq 0 \Longrightarrow b_{0,k}\neq 0$. Mutating at vertex $0$ yields the nwe B-matrix $\tilde{B}$ given by:
\begin{align*}
\tilde{b}_{i,j}= \left\{ \begin{array}{lc}
-b_{i,j}&i=0 \text{ or } j=0\\
b_{i,j}+b_{i,0} \cdot |b_{0,j}| \cdot \textbf{1}_{\{b_{0,i}b_{0,j}<0\}}& \text{otherwise}\end{array} \right.
\end{align*}

The B-matrix of the mutated quiver $\tau(Q)$ is
\begin{align*}
\tau(B)=\left( \begin{array}{ccccc}
0 & -b_{0,1} & -b_{0,2} & \ldots  &-b_{0,n-1}\\
-b_{1,0} & 0 & b_{1,2}+\epsilon_{1,2} & \ldots &b_{1,n-1}+\epsilon_{1,n-1}\\
-b_{2,0} & b_{2,1}+\epsilon_{2,1} & 0 & \ldots &b_{2,n-1}+\epsilon_{2,n-1}\\
\vdots   &\vdots & \vdots & \ddots &\vdots\\
-b_{n-1,0} & b_{n-1,1}+\epsilon_{n-1,1} & b_{n-1,2}+\epsilon_{n-1,2} & \ldots &0\\
\end{array} \right)
\end{align*}
where $\epsilon_{i,j}=b_{i,j}+b_{i,0} \cdot |b_{0,j}| \cdot \textbf{1}_{\{b_{0,i}b_{0,j}<0\}}$. The double quiver $Q$ has period 1 if $\tau(B)$ and $\mu B \mu^{-1}$ represent the same binomial seed, where $\mu$ is the permutation matrix such that $\mu B \mu^{-1}$ corresponds to the seed after the relabeling $(0,1,2,\ldots n-1) \rightarrow (n-1,0,1,\ldots n-2)$,
\begin{align*}
\mu B \mu^{-1}=\left( \begin{array}{ccccc}
0 & b_{n-1,0} & b_{n-1,1} & \ldots  &b_{n-1,n-2}\\
b_{0,n-1} & 0 & b_{0,1} & \ldots &b_{0,n-2}\\
b_{1,n-1} & b_{2,1}  & 0 & \ldots &b_{2,n-2}\\
\vdots   &\vdots & \vdots & \ddots &\vdots\\
b_{n-2,n-1} & b_{n-2,0} & b_{n-2,1} & \ldots &0\\
\end{array} \right)
\end{align*}
Therefore $Q$ is a period 1 double quiver if
\begin{align}
\label{cri}\tau(B)=\mu B\mu^{-1}.
\end{align}
Equivalently, $Q$ is a period 1 double quiver if
\begin{align}
\label{ex-3}b_{n-1,i}&=-b_{0,i+1}, &0\leq i\leq n-2,&\\
\label{ex-2}b_{i,n-1}&=-b_{i+1,0}, &0\leq i\leq n-2,& \text{ and}\\
\label{ex-1}b_{i,j}&=b_{i+1,j+1}+\epsilon_{i+1,j+1}, &0\leq i,j\leq n-2.&
\end{align}
Solving these equations leads to the following equations
\begin{eqnarray}
\notag -b_{i+1,0}=b_{i,n-1} &=& b_{i-1,n-2}-\epsilon_{i,n-1}\\
\notag &=& b_{i-2,n-3}-\epsilon_{i-1,n-2}-\epsilon_{i,n-1}\\
\notag &\vdots&\\
\label{ex1} &=& b_{0,n-i-1}-\epsilon_{1,n-i}-\epsilon_{2,n-i+1}-\ldots-\epsilon_{i,n-1}.\\
\notag -b_{0,i+1}=b_{n-1,i}&=& b_{n-2,i-1}-\epsilon_{n-1,i}\\
\notag  &=& b_{n-3,i-2}-\epsilon_{n-2,i-1}-\epsilon_{n-1,i}\\
\notag  &\vdots&\\
\label{ex2} &=& b_{n-i-1,0}-\epsilon_{n-i,1}-\epsilon_{n-i+1,2}-\ldots-\epsilon_{n-1,i}.
\end{eqnarray}

Using the same terminology as \cite{fm} for quivers, vertex $i$ of the double quiver $Q$ is said to be a {\it sink} if all the half-edges incident to $i$ are directed inwards. A double quiver is said to be a period 1 {\em sink-type} double quiver if vertex $0$ is sink, and the double quiver has period 1. From above, we can obtain the following theorem classifying all period 1 sink-type double quivers:

\begin{thm}\label{sinktypethm}
Let $B$ be the matrix of a sink-type double quiver $Q$. Then $Q$ is a period 1 double quiver if and only if the following conditions hold:
\begin{enumerate}
\item $b_{0,i}$ and $b_{0,n-i}$ are either both negative, or both zero, for $i=1,2\ldots n-1$.
\item $b_{i,0}=-b_{0,n-i}$, for $i=1,2,\ldots n-1$.
\item $b_{i,j}=b_{0,j-i}$ if $0<i<j\leq n-1$ and $b_{i,j}=-b_{0,n-i+j}$ if $0<j<i\leq n-1$.
\end{enumerate}
\begin{proof}
Since $Q$ is of sink type, then $b_{0,i}\leq 0$ for all $i$. Therefore $\epsilon_{i,j}=0$ for all $0<i,j\leq n-1$.

If all three conditions above are satisfied, then (\ref{ex-3}), (\ref{ex-2}) and (\ref{ex-1}) are trivially satisfied.

Conversely, let us assume $Q$ has period 1, so (\ref{ex-3}), (\ref{ex-2}) and (\ref{ex-1}) are satisfied. From (\ref{ex-1}), we have
\begin{align*}
b_{i,j}=b_{0,j-i}, \text{ if } 0 < i < j \leq n-1,\\
b_{i,j}=b_{i-j,0}, \text{ if } 0 < j < i \leq n-1.
\end{align*}
Combining with (\ref{ex-3}) and (\ref{ex-2}), we have,
\begin{align*}
b_{i,n-1}=b_{0,n-i-1}=-b_{i+1,0},\\
b_{n-1,i}=-b_{0,i+1}=b_{n-i-1,0}.
\end{align*}
It follows that $b_{0,i+1}=0 \Longrightarrow b_{n-i-1,0}=0$. Since $0$ is a mutable vertex, we have $b_{0,n-i-1}=0$ for all $i$.
\end{proof}
\end{thm}

The seed $t = \bf (x, p)$ corresponding to a sink-type double quiver $Q$ is such that $p_0$ is of the form $\prod_{j}{x_j^{a_j}} + 1$. If $Q$ has period 1, then $a_i = 0 \Longleftrightarrow a_{n-i} = 0$ follows from (1) in the theorem above. Conversely, any polynomial of this form generates a period 1 seed as item (2) of Theorem \ref{families} shows. Thus Theorem \ref{sinktypethm} can be restated as:

\begin{thm}\label{actualsinktypethm}
The only period 1 binomials such that the quiver corresponding to $\bf (x, P)$ is of sink-type and has period 1 are those of the form $\displaystyle P = x_1^{a_1}x_2^{a_2}\cdots x_{n-1}^{a_{n-1}} + 1$, where $a_i = 0 \Longleftrightarrow a_{n-i} = 0$, for all $i$.
\end{thm}

\subsection{Mutual Double Quiver}
In general, given a binomial seed (all the exchange polynomials are binomials), the corresponding double quiver is not unique. For example, if we reverse all the half-edges attached at a certain vertex, the new double quiver represents the same seed. However, there is a canonical choice, which coincides with usual quivers (if we regard a quiver as a double quiver, as it was done in (1) of Remark \ref{quivervsdouble}). From the mutability of vertex $0$, if there are half-edges between $0$ and $i$ attached to $i$, then there are half-edges between $0$ and $i$ attached to $0$. We make all these half-edges point in the same direction by reversing all the half-edges attached at $i$, if necessary. In terms of the B-matrix, the resulting canonical quiver is such that $b_{0,i}$ and $b_{i,0}$ are of opposite sign. A double quiver with this condition is said to be {\it mutual} at vertex $0$. Such double quiver is said to be the {\it canonical double quiver} associated to the seed. In this subsection, we prove Theorem \ref{mutualdoublequiver} regarding period 1 mutual (at $0$) double quivers. By translating this into the language of period 1 polynomials, this is equivalent to Theorem \ref{binomialthm} in Section \ref{results}.

\begin{thm}\label{mutualdoublequiver}
Let $B$ be the matrix associated to a canonical mutual double quiver $Q$. Then $Q$ has period 1 if and only if the following conditions hold:
\begin{enumerate}
\item $b_{i,0}$ and $b_{0,i}$ are of opposite signs, or both zero, for $0 < i\leq n-1$.
\item $b_{i,0}=-b_{0,n-i}$, for $i=1,2,\ldots n-1$.
\item $b_{i,j}=-\sum_{k=0}^{i}\epsilon_{i-k,j-k}+b_{0,j-i}$ if $0<i<j\leq n-1$.
\item $b_{i,j}=-\sum_{k=0}^{j}\epsilon_{i-k,j-k}-b_{0,n-i+j}$ if $0<j<i\leq n-1$.
\end{enumerate}
\end{thm}
\begin{proof}
We assume $Q$ is a canonical mutual double quiver. If all four conditions above are satisfied, then (\ref{ex-3}), (\ref{ex-2}) and (\ref{ex-1}) are trivially satisfied.

Conversely, let us assume $Q$ has period 1, so (\ref{ex-3}), (\ref{ex-2}) and (\ref{ex-1}) are satisfied. We first prove by induction
\begin{align}
\label{ex3}b_{0,i}=-b_{n-i,0} \text{ and } b_{i,0}=-b_{0,n-i}.
\end{align}
Setting $i=0$ in (\ref{ex1}) and (\ref{ex2}) gives the base cases $b_{0,1}=-b_{n-1,0}$ and $b_{1,0}=-b_{0,n-1}$. Now, assume that (\ref{ex3}) holds for $i=0,1,2\ldots k$;  we prove it for $i=k+1$. Note that,
\begin{eqnarray}
\notag-b_{k+1,0} &=& b_{0,n-k-1}-\sum_{j=1}^{k}\epsilon_{j,n-k+j-1}\\
\notag &=& b_{0,n-k-1}-\sum_{j=1}^{k}b_{j,0} \cdot |b_{0,n-k+j-1}| \cdot \textbf{1}_{\{b_{0,j}b_{0,n-k+j-1}<0\}}\\
\label{ex3.5} &=& b_{0,n-k-1}-\frac{1}{2}\cdot\big(\sum_{j=1}^{k}b_{j,0} \cdot |b_{0,n-k+j-1}| \cdot \textbf{1}_{\{b_{0,j}b_{0,n-k+j-1}<0\}}\\
\notag &+& \sum_{j=1}^{k}b_{k-j+1,0} \cdot |b_{0,n-j}| \cdot \textbf{1}_{\{b_{0,k-j+1}b_{0,n-j}<0\}}\big)
\end{eqnarray}
From the inductive hypothesis, $b_{j,0}=b_{0,n-j}$ and $b_{k-j+1,0}=b_{0,n-k+j-1}$, so,
\begin{eqnarray}
\notag b_{j,0} \cdot |b_{0,n-k+j-1}| \cdot \textbf{1}_{\{b_{0,j}b_{0,n-k+j-1}<0\}}
+b_{k-j+1,0} \cdot |b_{0,n-j}| \cdot \textbf{1}_{\{b_{0,k-j+1}b_{0,n-j}<0\}}\\
\label{ex4}= b_{0,n-j} \cdot |b_{0,n-k+j-1}| \cdot \textbf{1}_{\{b_{0,j}b_{k-j+1,0}<0\}}
+b_{0,n-k+j-1} \cdot |b_{0,n-j}| \cdot \textbf{1}_{\{b_{0,k-j+1}b_{j,0}<0\}}
\end{eqnarray}
From the mutuality assumption, $\textbf{1}_{\{b_{0,k-j+1}b_{j,0}<0\}}=\textbf{1}_{\{b_{0,j}b_{k-j+1,0}<0\}}$. If both of them are $0$, then (\ref{ex4}) is zero. If both of them are $1$, then  $b_{0,n-k+j-1}$ and $b_{0,n-j}$ are of opposite sign, therefore (\ref{ex4}) is zero. Substituting back into (\ref{ex3.5}), we have $-b_{k+1,0}=b_{0,n-k-1}$. From (\ref{ex2}), by a similar argument, we obtain $-b_{0,k+1} = b_{n-k-1,0}$.

Finally, conditions (3) and (4) follow from equations (\ref{ex3}) and (\ref{ex-3}).
\end{proof}

\begin{rem}
Theorem \ref{binomialthm} generalizes the main theorem of \cite{fm}.
For example, the polynomial $P = x_1^ax_2^b + 1$, with $a\neq b$, produces Laurent phenomenon sequences; it is predicted by Theorem \ref{binomialthm}, but not by \cite[Theorem 6.6]{fm}.
\end{rem}

\begin{rem}
A double quiver with a skew-symmetrizable matrix B-matrix is mutual at each vertex. Hence, restricting $B$ to be skew-symmetric, Theorem \ref{mutualdoublequiver} provides a classification of period 1 cluster algebras over the coefficient ring $\Z$.
\end{rem}

\section{Classification of period 1 Seeds for Small $n$}\label{smallnsec}

In this section we prove the classifications stated in Section 3 of all period 1 seeds when $n=2$ and $n=3$.

\subsection{Proof of Theorem \ref{n2thm}}
Let $P = P(x_1) = \sum_{i = 0}^d{a_ix_1^i}$ be an (irreducible) polynomial of degree $d>0$ that generates a period 1 seed. Since $P$ is not divisible by $x_1$, we know $a_0 \neq 0$ and $P|_{x_1 = 0} = a_0 \neq 0$. From the definition of $\tau$, $P$ generates a period 1 seed if and only if
\begin{align}\label{ntwo} P(x_1) = a_0^{-1}x_1^d\cdot\sum_{i = 0}^d{a_i\big( \frac{a_0}{x_1} \big)^i} = \sum_{i = 0}^d{(a_ia_0^{i-1})x_1^{d - i}}.\end{align}

For each $0 \leq i \leq d$, by equating the coefficient of $x_1^i$ on both sides of (\ref{ntwo}), we see that
\begin{align} \label{equate} a_i = a_{d-i}a_0^{d-i-1} \textrm{ for all $i$}. \end{align}
In particular, when $i = d$, we obtain $a_d = 1$, so $P$ has to be monic.

If $d > 2$, we have that $a_i = a_{d-i}a_0^{d-i-1} = (a_ia_0^{i-1})a_0^{d-i-1} = a_ia_0^{d-2}$ for all $i$. Setting $i=d$ gives that $a_0 = \pm 1$. If $a_0 = 1$ then (\ref{equate}) implies $a_i = a_{d - i}$ for all $i$, or equivalently, $P$ is palindromic, and all such polynomials satisfy (\ref{ntwo}). If $a_0 = -1$, then (\ref{equate}) implies $a_i = (-1)^{d-i-1}a_{d - i}$ for all $i$. When $d$ is odd, these relations when $i=i_0, d - i_0$ imply $a_{i_0} = 0$ for all $1 \leq i_0 \leq d-1$. But then $P = x_1^d - 1$ is not irreducible. When $d$ is even, we find that $a_i = - a_{d-i}$ for all $i$, or equivalently, $P$ is antipalindromic.

If $d = 2$, (\ref{equate}) is trivially satisfied for any $a_1, a_2$.

If $d = 1$, (\ref{equate}) with $i = 0$ gives $a_0 = a_1 = 1$.

\subsection{Proof of Theorem \ref{n3thm}}
\subsubsection{Bounding Degrees}
Assume the following is a period 1 seed
\begin{align*}
\big{\{}x_0, P(x_1,x_2)\big{\}},\\
\big{\{}x_1, Q(x_0,x_2)\big{\}},\\
\big{\{}x_2, P(x_0,x_1)\big{\}},
\end{align*}
where $P$ is a two-variable irreducible polynomial not divisible by $x_1$ or $x_2$. By Proposition \ref{irreducible}, $Q(x_0,x_2)$ is also irreducible and not divisible by $x_0$ or $x_2$. Then $P(x_0,0)$ and $P(0,x_1)$ are not $0$. It is not hard to see that there do not exist period 1 polynomials $P$ that do not depend on $x_0$ or $x_1$. Thus assume that $P$ depends on both of these variables.
Let $m$ be the degree of $x_0$ in $P(x_0,x_1)$; we can write
\begin{align}
\label{DP}P(x_0,x_1)=\sum_{k=0}^{m}f_{k}(x_1)x_0^k,
\end{align}
where the $f_k$ are single variable polynomials for $k=0,1,\ldots m$. Let $\tilde{Q}$ be the intermediate polynomial $G$ at step (2) of applying $\tau$ to $Q$. Then
\begin{align*}
\tilde{Q}(x_0,x_2)=P\left(\frac{P(x_0,0)}{x_2},x_0\right)=\sum_{k=0}^{m}f_k(x_0)\cdot\frac{P^{k}(x_0,0)}{x_2^k}.
\end{align*}
Let $d(x_0)$ be the maximal factor of $\tilde{Q}(x_0,x_2)$, which is in the form $x_0^{s_1}\cdot p(x_0)$, where $f(x_0)$ is a factor of $P(x_0,0)^K$ for some $K$. From the rules for computing $\tau$, we have
\begin{align}
\label{Q}Q(x_0,x_2)=\sum_{k=0}^{m}\frac{f_k(x_0)P^{k}(x_0,0)}{d(x_0)}x_2^{m-k}.
\end{align}
In view of (\ref{Q}), the coefficient of $x_2^m$ in $Q$ is \[ \frac{f_0(x_0)}{d(x_0)} = \frac{P(0, x_0)}{d(x_0)}, \] which is a non-vanishing polynomial. Therefore $d(x_0)$ divides $f_0(x_0)=P(x_0,0)$, and $x_2$ is of degree $m$ in $Q(x_0,x_2)$.

We can similarly obtain $P(x_1, x_2)$ from $Q(x_0, x_2)$. Let $n$ be the degree of $x_0$ in $Q(x_0,x_2)$; we can write
\begin{align}
\label{P}P(x_1,x_2)=\sum_{k=0}^{n}\frac{g_k(x_1)P^{k}(0,x_1)}{t(x_1)}x_2^{n-k}.
\end{align}
The coefficient of $x_2^n$ in $P$ is \[ \frac{g_0(x_1)}{t(x_1)} = \frac{P(x_1, 0)}{t(x_1)}, \] which is a nonzero polynomial. Therefore $t(x_1)$ divides $g_0(x_1)=Q(0,x_1)$ and the degree of $x_2$ in $P(x_1,x_2)$ is $n$.

If we let $x_2=0$ in (\ref{Q}), then
\begin{align}
\label{LQ}Q(x_0,0)=\frac{f_m(x_0)P^{m}(x_0,0)}{d(x_0)}
\end{align}
and similarly
\begin{align}
\label{LP}P(x_0,0)=\frac{g_n(x_0)P^{n}(0,x_0)}{t(x_0)}.
\end{align}
Comparing the degree of both sides of (\ref{LQ}) and (\ref{LP}) and recalling the divisibility relations $d(x_0) \mid P(0,x_0)$, $t(x_0) \mid Q(0,x_0)$, we arrive at the inequalities
\begin{align}
\label{bound1}\deg f_m(x_0)+m\deg P(x_0,0)=&\deg Q(x_0,0)+\deg d(x_0)\leq \deg Q(x_0,0)+\deg P(0,x_0)\\
\label{bound2}\deg g_n(x_0)+n\deg P(0,x_0)=&\deg P(x_0,0)+\deg t(x_0)\leq \deg P(x_0,0)+\deg Q(0,x_0)
\end{align}
Summing (\ref{bound1}) and (\ref{bound2}), and noticing that $\deg Q(0,x_0)\leq m$, $\deg Q(x_0,0)\leq n$, we obtain the following inequality:
\begin{align}\label{bound}
2\geq \deg g_n(x_0)+\deg f_m(x_0)+(m-1)\big{(}\deg P(x_0,0)-1\big{)}+(n-1)\big{(}\deg P(0,x_0)-1\big{)}.
\end{align}

From this inequality, the classification of period 1 polynomials is decomposed into the following five cases:
\begin{enumerate}

\item $\deg P(x_0,0)=2$ and $\deg P(0,x_0)=1$
\item $\deg P(x_0,0)=1$ and $\deg P(0,x_0)=2$
\item $\deg P(x_0,0)=\deg P(0,x_0)=1$
\item $\deg P(x_0,0)=\deg P(0,x_0)=2$
\item Either $\deg P(x_0,0)=0$ or $\deg P(0,x_0)=0$
\end{enumerate}

\begin{lem}\label{lemma1}
If $P(0,0) \neq 0$, then the bound (\ref{bound2}) can be refined to
\begin{align}
\label{bound3}\deg g_n(x_0)+n\deg P(0,x_0)=&\deg P(x_0,0)+\deg t(x_0)\leq 2\deg P(x_0,0).
\end{align}
\end{lem}
\begin{proof}
Since $d(x_0)  \mid  P(0,x_0)$ and $x_0 \nmid P(0,x_0)$, we see that $x_0 \nmid d(x_0)$.

We moreover claim that $x_1 \nmid t(x_1)$. Assume otherwise that $x_1  \mid  t(x_1)$. From (\ref{P}), $t(x_1)$ divides $g_k(x_1)P^{k}(0,x_1)$ for all $k=0,1,2\ldots n$ and $x_1 \nmid P(0,x_1)$. Therefore $x_1  \mid  g_k(x_1)$ for all $k=0,1,\ldots n$. This implies that $x_1  \mid  Q(x_0,x_1)$, which contradicts the irreducibility of $Q(x_0,x_1)$.

Since $P(x_0,0)=\sum_{k=0}^{m}f_k(0)x_0^k$, then $f_k(0)=0$ if $k>\deg p(x_0,0)$. Hence
\begin{align*}
t(x_1) \mid Q(0,x_1)=&\sum_{k=0}^{m}\frac{f_k(0)P^{k}(0,0)}{d(0)}x_1^{m-k}\\
               =&\left(\sum_{k=0}^{\deg P(x_0,0)}\frac{f_k(0)P^{k}(0,0)}{d(0)}x_1^{\deg P(x_0,0)-k}\right)\cdot x_1^{m-\deg P(x_0,0)}.
\end{align*}
Therefore $\deg t(x_1)\leq \deg P(x_0,0)$ and (\ref{bound2}) leads to the desired inequality
\begin{align*}
\deg g_n(x_0)+n\deg P(0,x_0)=&\deg P(x_0,0)+\deg t(x_0)\leq 2\deg P(x_0,0).
\end{align*}
\end{proof}

\begin{lem}\label{lemma2}
If $f_m(x_1)$ has a nonzero constant term, then $m = \deg P(x_0,0)$.
\end{lem}
\begin{proof}
In view of (\ref{P}), $P(x_0,x_1)$ contains the term $f_m(x_1)x_0^m$. Since $f_m(x_1)$ contains a nonzero constant term, then $f_m(0)\neq 0$. Thus, $P(x_0,0)$ contains the term $f_m(0)x_0^m$. This implies $\deg P(x_0,0)\geq m$ and since $m$ is the degree of $x_0$ in $P(x_0,x_1)$, we have $\deg P(x_0,0)\leq m$. Therefore $m=\deg P(x_0,0)$.
\end{proof}
\subsubsection{Analysis of Cases}

\begin{enumerate}[{Case} 1]
\item $\deg P(x_0, 0) = 2$ and $\deg P(0, x_0) = 1$.

From (\ref{bound}), we either have $m=2$ or $m=3$.

If $m=3$, we obtain $\deg f_3(x_1)=0$ from (\ref{bound}). However, Lemma \ref{lemma2} says that $P(x_0,0) = 3$, thus implying $f_3(x_1) \neq 0$, a contradiction. 

If $m=2$, (\ref{bound1}) tells us that
$\deg f_2(x_0) + 4 = \ deg Q(x_0,0) + \deg d(x_0)\leq n+1$, thus $n\geq 3$. From (\ref{DP}), we have
\begin{align*}
P(x_0,x_1)=f_2(x_1)x_0^2+f_1(x_1)x_0+f_0(x_1).
\end{align*}
Since the degree of $x_1$ in $P(x_0,x_1)$ is $n\geq 3$, $\deg f_2(x_1)\leq 1$ and $\deg f_0(x_1)=2$, we must have $\deg f_1(x_1)=n$. In view of (\ref{P}),
\begin{align*}
Q(x_0,x_2)=&\frac{f_0(x_0)}{d(x_0)}x_2^2+\frac{f_1(x_0)P(x_0,0)}{d(x_0)}x_2+\frac{f_2(x_0)P^2(x_0,0)}{d(x_0)}.
\end{align*}
As the degree of $x_0$ in $Q(x_0, x_2)$ is $n$, we have \[ n \geq \deg_{x_0}\left(\frac{f_1(x_0)P(x_0,0)}{d(x_0)}\right. \] Then
\begin{align*}
n\geq \deg f_1(x_0)+\deg P(x_0,0)-\deg d(x_0)\geq n+2-1=n+1.
\end{align*}
Therefore there are no period 1 polynomials in this case.

\item $\deg P(x_0, 0) = 1$ and $\deg P(0, x_0) = 2$.

From (\ref{bound}), we either have $n=2$ or $n=3$.

If $n=3$, we obtain $\deg f_m(x_1)=0$ from (\ref{bound}). From Lemma (\ref{lemma2}), we have $m=\deg P(x_0,0)=1$. Hence in (\ref{DP}),
\begin{align*}
P(x_0,x_1)=f_1(x_1)x_0+f_0(x_1).
\end{align*}
However, $\deg f_1(x_1)=0$ and $\deg f_0(x_1)=2$, so the degree of $x_1$ in $P(x_0,x_1)$ is $2 \neq n$, a contradiction.

If $n=2$, inequalities (\ref{bound1}) and (\ref{bound2}) yield
\begin{align*}
\deg f_m(x_0)+m\leq 4,\\
\deg g_n(x_0)+3\leq m.
\end{align*}
Thus $m\geq 3>1=\deg P(x_0,0)$. From Lemma (\ref{lemma2}), we have $\deg f_m(x_1)\geq 1$. These inequalities yield $m=3$, $\deg f_3(x_1)=1$ and $\deg g_n(x_0) = 0$.
Moreover, from (\ref{bound2}), $\deg t(x_0)=3$. Since these values do not satisfy (\ref{bound3}), Lemma (\ref{lemma1}) tells us that $P(0,0) = 0$, i.e., $P(x_0, x_1)$ does not have a constant term.
Taking $m=3$ in (\ref{Q}) yields
\begin{align*}
Q(x_0,x_2)=\frac{f_0(x_0)}{d(x_0)}x_2^3+\frac{f_1(x_0)P(x_0,0)}{d(x_0)}x_2^2+\frac{f_2(x_0)P^2(x_0,0)}{d(x_0)}x_2^2+\frac{f_3(x_0)P^3(x_0,0)}{d(x_0)}.
\end{align*}
From (\ref{bound1}) and the values already found, we also have $\deg d(x_0) = 2$. Since $f_3(x_0)$ and $P(x_0,0)$ are both linear polynomials without constant terms, $f_3(x_0)P^3(x_0,0)=ax_0^4$ for some $a\in\Z$. Since $d$ has degree $2$ and divides $(x_0) \mid f_3(x_0)P^3(x_0,0)$, then $d(x_0)=bx_0^2$ for some $b\in\Z, b\neq 0$. From $d(x_0) \mid f_0(x_0)$ and $d(x_0) \mid f_1(x_0)P(x_0,0)$, we have $x_0 \mid f_0(x_0)$ and $x_0 \mid f_1(x_0)$. From equation (\ref{DP}), we have
\begin{align*}
P(x_0,0)=f_3(0)x_0^3+f_2(0)x_0^2+f_1(0)x_0+f_0(0)=f_3(0)x_0^3+f_2(0)x_0^2.
\end{align*}
Since in this case, $\deg P(x_0,0)=1$, then $P(x_0,0)=0$. This contradicts the fact that $P(x_0,x_1)$ is not divisible by $x_1$.

\item $\deg P(x_0, 0) = \deg P(0, x_0) = 1$.

If $P(x_0,x_1)$ contains no constant term, then we can write
\begin{align*}
P(x_1,x_2) = ax_1+bx_2+x_1x_2R(x_1,x_2),
\end{align*}
where $a, b\in\Z$ are both nonzero and $R(x_1,x_2)$ is a polynomial of degree $m-1$ in $x_1$ and degree $n-1$ in $x_2$. We next obtain $\tau_P(P(x_0, x_1))$ by replacing $x_0$ with $\frac{ax_1}{x_3}$, downshifting and then multiplying by a monomial $M$. This monomial has to be such that the resulting $Q$ is a Laurent polynomial, not divisible by any $x_i$ and its coefficients have greatest common divisor $1$. Assume $\tilde{M}$ is $M$, but with coefficient $1$ and let $\tilde{Q}$ be the resulting polynomial. Thus $\tilde{Q} = cQ$ for some constant $c$. We can write $\tilde{Q}$ as
\begin{align*}
\tilde{Q}(x_0,x_2)=a^2x_2^{m-1}+bx_2^m+ax_0R(\frac{ax_0}{x_2},x_0)x_2^{m-1}
\end{align*}
From this polynomial, we analogously obtain $\tilde{P}$ by omitting a constant factor for the {\it adjusting monomial}
\begin{align*}
\tilde{P}(x_1,x_2)=a^2x_2^{k+1}+bx_1x_2^{k+1}+abx_1R(\frac{ab}{x_2},\frac{bx_1}{x_2})x_2^k,
\end{align*}
where $k$ is the least integer for which $abx_1R(\frac{ab}{x_2},\frac{bx_1}{x_2})x_2^k$ is a polynomial. In view of the above equation, $k+1=\deg \tilde{P}(0,x_2)=P(0,x_2)=1$. Therefore $k=0$ and so $R(x_0,x_1)$ is a contant; write $R = R(x_0,x_1)$. Then
$\displaystyle\tilde{P}(x_1,x_2) = abRx_1+a^2x_2+bx_1x_2$
must be equal to $\lambda P(x_1,x_2)=\lambda ax_1+\lambda bx_2+\lambda Rx_1x_2$, where $\lambda$ is a nonzero integer. After equating coefficients, we obtain the polynomials in items (1) and (2) of Theorem \ref{n3thm} whose generated seeds are
\begin{eqnarray*}
\{x_0, ax_1+ax_2+x_1x_2\}, \{x_1, a+x_0+x_2\}, \{x_2, ax_0+ax_1+x_0x_1\}
\end{eqnarray*}
and
\begin{eqnarray*}
\{x_0, ax_1-ax_2+x_1x_2\}, \{x_1, a+x_0-x_2\}, \{x_2, ax_0-ax_1+x_0x_1\}.
\end{eqnarray*}

If $P(x_0,x_1)$ has a nonzero constant term, we can write
\begin{align*}
P(x_0,x_1)=c+ax_0+bx_1+x_0x_1R(x_0, x_1).
\end{align*}
From Lemma \ref{lemma1}, we have $n\leq 2$. Substituting into (\ref{bound1}),
\begin{align*}
\deg f_m(x_0)+m\leq \deg d(x_0) +2.
\end{align*}
If $a\neq b$, then $\gcd (P(x_0,0),P(0,x_0))=1$, so $\deg d(x_0)=0$. We can then write constant $d(x_0)$ as $d$. By the same argument as in Lemma \ref{lemma2}, we have $m=1$. Then (\ref{Q}) reads
\begin{align*}
Q(x_0,x_2)=\frac{f_1(x_0)P(x_0,0)}{d}+\frac{f_0(x_0)}{d}x_2,
\end{align*}
and since the degree of $x_0$ in $Q(x_0,x_2)$ is $n\leq 2$, we must have $\deg f_1(x_0) \leq 1$. Then we can write $P(x_0,x_1)=c+ax_0+bx_1+Rx_0x_1$ for some $R\in\Z$ and where $a\neq b$. By the same argument as above, we obtain the period 1 polynomials in items (3) and (4) of Theorem \ref{n3thm} whose generated seeds are
\begin{eqnarray*}
\{x_0, x_1-x_2-1\}, \{x_1, -x_0x_2+x_0-x_2-1\}, \{x_2, x_0-x_1-1\}
\end{eqnarray*}
and
\begin{eqnarray*}
\{x_0, -x_1+x_2-1\big{\}}, \{x_1, x_0x_2+x_0-x_2+1\}, \{x_2, -x_0+x_1-1\}.
\end{eqnarray*}
If $a=b$ then $m=1$ or $2$. If $m=1$, from (\ref{P}) and (\ref{Q}),
\begin{align*}
Q(x_0,x_2)=f_1(x_0)+x_2,\qquad P(x_1,x_2)=\frac{x_2^n}{t(x_1)}\left(f_1(\frac{f_0(x_1)}{x_2})+x_1\right).
\end{align*}
Since $f_1(x_0)$ has constant term $a$,  $P(x_1,x_2)$ contains the term $x_2^n$. Therefore $n=1=\deg f_1(x_0)$. We can then write $P(x_0,x_1) = c + ax_0 + ax_1 + Rx_0x_1$. The same argument as above yields the period 1 polynomials in (5) of Theorem \ref{n3thm}, whose associated seed is
\begin{align*}
\big{\{}x_0&, x_1x_2+ax_1+ax_2+c\big{\}},\\
\big{\{}x_1&, x_0+x_2+a\big{\}},\\
\big{\{}x_2&, x_0x_1+ax_0+ax_1+c\big{\}}.
\end{align*}
If $m=2$, then $m\neq \deg P(x_0,0)$. From Lemma \ref{lemma1}, $f_{m}$ does not contain constant term, so $\deg f_{m}\geq 1$. Since $\deg d\leq \deg P(0,x_0)=1$, plug them into (\ref{bound1}), we get $n=2$ and $\deg f_m=1$.  Morover, from Lemma \ref{lemma2}, $\deg t\leq 1$. Plug them into (\ref{bound2}), we get $\deg g_n=0$. From (\ref{DP})
\begin{align*}
P(x_0,x_1)=f_2(x_1)x_0^2+f_1(x_1)x_0+f_0(x_1),
\end{align*}
and the degree of $x_1$ is $2$, it must be that $\deg f_1=2$. Mutating at $x_0$ to obtain $Q(x_0,x_2)$ gives (notice $P(x_0,0)=P(0,x_0)$),
\begin{align*}
Q(x_0,x_2)=f_2(x_0)P(x_2,0)+f_1(x_0)x_2+x_2^2.
\end{align*}
This contains the term $x_0^2x_2$, which contradicts that $\deg g_2=0$.

\item $\deg P(x_0, 0) = \deg P(0, x_0) = 2$.

From (\ref{bound3}), we see that $m=n=\deg P(x_0,0)=\deg P(0,x_0)=2$ and $\deg g_n(x_0)=\deg f_m(x_0)=0$. Hence,
\begin{align}
\notag P(x_0,x_1)=&f_2(x_1)x_0^2+f_1(x_1)x_0+f_0(x_1),\\
\label{Q2}Q(x_0,x_2)=&\frac{f_2(x_0)P^2(x_0,0)}{d(x_0)}+\frac{f_1(x_0)P(x_0,0)}{d(x_0)}x_2+\frac{f_0(x_0)}{d(x_0)}x_2^2.
\end{align}
The first equation gives $\deg f_0(x_1) = \deg P(0, x_1) = 2$.
From the second, by looking at the coefficient of $x_2^2$, we have $\deg d(x_0) \leq \deg f_0(x_0) = 2$. Moreover, remember we had $Q(x_0, x_2) = g_2(x_2)x_0^2 + g_1(x_2)x_0 + g_0(x_2)$. Since $\deg g_2(x_0) = 0$, $Q$ does not contain terms divisible by $x_2x_0^2$. By looking at the coefficient of $x_2$ in $Q$ in equation (\ref{Q2}), we have that $\deg f_1(x_0) + \deg P(x_0, 0) - \deg d(x_0) \leq 1$, from which $\deg f_1(x_0) \leq 1$.
Since $\deg f_2(x_0)=0$, from (\ref{Q2}) we have that $d(x_0) \mid P^{2}(x_0,0)$.

If $d(x_0)\nmid P(x_0,0)$, then $d(x_0)=cr^2(x_0)$ for some monic linear factor $r(x_0)$ and constant $c\in\Z$. We will omit the constant $c$ as it will factor later, so simply write $d(x_0) = r^2(x_0)$. In particular, we have $\deg d(x_0) = 2$, from which $\deg f_0(x_0)  = 2$. Moreover, since $d(x_0)  \mid  f_1(x_0)P(x_0, 0)$, then $\deg f_1(x_0) = 1$. From the divisibility relations, we can write
\begin{align*}
&P(x_0,0)=t(x_0)r(x_0),\qquad  f_0(x_0)=Ar^2(x_0),\\ &f_1(x_0)=Br(x_0), \qquad f_2(x_0)=C,
\end{align*}
for some polynomial $t$ and constants $A, B, C\in\Z$.
Expression (\ref{Q2}) can then be simplified:
\begin{align*}
Q(x_0,x_2)=Ax_2^2 + Bt(x_0)x_2 + Ct^2(x_0)
\end{align*}
We also have
\begin{align*}
Ar^2(x_2) = P(0,x_2) = Ax_2^2 + Bt(0)x_2 + Ct^2(0).
\end{align*}
From both equations, we have
\begin{align*}
t^2(0)Q(x_0,x_2) = Ar^2 \left(\frac{x_2t(0)}{t(x_0)} \right)t^2(x_0).
\end{align*}
Since $Q(x_0,x_2)$ is irreducible, then $t(0)=0$. From above, we have $Ar^2(x_2) = Ax_2^2$ and so $d(x_0) = r^2(x_0) = x_0^2$. Moreover, $t(0) = 0$ implies $x_0  \mid  t(x_0)$ and so $d(x_0) =x_0^2  \mid  t(x_0)r(x_0) = P(x_0, 0)$. This is a contradiction with our initial assumption.

Now assume $d(x_0)  \mid  P(x_0, 0)$.

If $P(x_0,x_1)$ has a nonzero constant term, then since $d(x_0) \mid P(x_0,0)$ and $d(x_0) \mid P(0,x_0)$, we see $P(x_0,0)=P(0,x_0)$. $P(x_0,x_1)$ must be of the form
$\displaystyle P(x_0,x_1)=ax_0^2+ax_1^2+bx_0x_1+cx_0+cx_1+d$.
In this case, we obtain the period 1 polynomials in item (6) of Theorem \ref{n3thm}, whose generated seeds are
\begin{align*}
\{x_0&, x_1^2+x_2^2+ax_1x_2+bx_1+bx_2+c\},\\
\{x_1&, x_0^2+x_2^2+ax_0x_2+bx_0+bx_2+c\},\\
\{x_2&, x_0^2+x_1^2+ax_0x_1+bx_0+bx_1+c\},
\end{align*}
and
\begin{eqnarray*}
\{x_0, -x_1^2-x_2^2+ax_1x_2+c\}, \{x_1, x_0^2+x_2^2+ax_0x_2-c\}, \{x_2, -x_0^2-x_1^2+ax_0x_1+c\}.
\end{eqnarray*}
If $P(x_0,x_1)$ does not have a constant term, then $d(x_0)=ax_0^2+bx_0$ and $P(x_0,x_1)$ is of form $P(x_0,x_1) = d_1(ax_0^2+bx_0)+d_2(ax_1^2+bx_1)+cx_0x_1$.
In this case, we obtain a special case of (6) and the general polynomial in (7) of Theorem \ref{n3thm}. Their generated seeds are
\begin{align*}
\{x_0&, x_1^2+x_2^2+ax_1x_2+bx_1+bx_2\},\\
\{x_1&, x_0^2+x_2^2+ax_0x_2+bx_0+bx_2\},\\
\{x_2&, x_0^2+x_1^2+ax_0x_1+bx_0+bx_1\}
\end{align*}
and
\begin{eqnarray*}
\{x_0, -x_1^2-x_2^2+ax_1x_2\}, \{x_1, x_0^2+x_2^2-ax_0x_2\},
\{x_2, -x_0^2-x_1^2+ax_0x_1\}.
\end{eqnarray*}

\item Either $\deg P(x_0, 0) = 0$ or $\deg P(0, x_0) = 0$.

From Lemma (\ref{lemma1}), $\deg P(x_0,0)=0 \Longrightarrow \deg P(0,x_0)=0$. Thus, we only consider the case where $\deg P(0,x_0)=0$, i.e., $P(0, x_0)$ is a nonzero constant $a\in\Z$. Observe that $d(x_0) \mid P(0,x_0)$ implies that $d(x_0)$ is a constant $d$. Equation (\ref{P}) can then be simplified to
\begin{align}
\label{case5.1}P(x_1,x_2)=\sum_{k=0}^{n}\frac{g_k(x_1)a^k}{t(x_1)}x_2^{n-k},
\end{align}
from which $t(x_1) \mid g_k(x_1)$ for all $k$. Therefore $t(x_1) \mid Q(x_0,x_1)$. Since $Q(x_0,x_1)$ is irreducible, $t(x_1)$ is a constant $t$. Equation (\ref{case5.1}) with $x_2 = 0$ and $x_1 = 0$ yield
\begin{eqnarray*}
P(0,x_2) &=& a = \sum_{k=0}^{n}\frac{g_k(0)a^k}{t(x_1)}x_2^{n-k},\\
P(x_1,0) &=&\frac{a^{n}g_{n}(x_1)}{t}= \frac{ag_n(x_1)}{g_n(0)}.
\end{eqnarray*}

From the first one, we have $t=g_n(0)a^{n-1}$ and $x_1 \mid g_k(x_1)$ for $0 \leq k\leq n-1$. Now mutating $P(x_1,x_2)$ at $x_0$ gives,
\begin{align}
\notag Q(x_0,x_2)=&\frac{x_2^m}{d}P(\frac{P(x_0,0)}{x_2},x_0)\\
\label{done1}          =&\frac{x_2^m}{td}\sum_{k=0}^{n}g_k(\frac{ag_n(x_0)}{g_n(0)x_2})a^kx_0^{n-k}\\
\label{done2}          =&\sum_{k=0}^{n}g_k(x_2)x_0^k\\
\notag          =&\sum_{l=0}^{m}h_k(x_0)x_2^l.
\end{align}
Next we compute $h_m$. Since $x_2 \mid g_k(x_2)$ for $k=0,1,2\ldots, n-1$, in (\ref{done1}), only the term $\frac{x_2^m}{td}g_n(\frac{ag_n(x_0)}{g_n(0)x_2})a^n$ contains term $x_2^m$. Indeed,
\begin{align*}
h_m(x_0)x_2^m=\frac{g_n(0)a^n}{td}x_2^m.
\end{align*}
Therefore $h_m(x_0)$ is a constant, so only $g_0(x_2)$ contains the term $x_2^m$, say $g_0(x_2)=bx_2^m+\ldots$. Since $\deg g_k<m$ for $k=1,2,\ldots, n-1$, 
\begin{align*}
x_2 \mid \frac{x_2^m}{td}g_k(\frac{ag_n(x_0)}{g_n(0)x_2})a^kx_0^{n-k},\quad k=1,2,\ldots n-1.
\end{align*}
Setting $x_2=0$, the above expressions all vanish, so
\begin{align*}
Q(x_0,0)=h_0(x_0)=\frac{b}{td}\Big{(}\frac{ag_n(x_0)}{g_n(0)}\big{)}^{m}x_0^n.
\end{align*}
Since $\deg Q(x_0,0)\leq n$, $\deg g_n(x_0)=0$ (or $m=0$, then $P(x_0,x_1)$ only depends on $x_0$). (\ref{done1}) is simplified as
\begin{align*}
\frac{x_2^m}{td}\sum_{k=0}^{n}g_k(\frac{a}{x_2})a^kx_0^{n-k}
=\sum_{k=0}^{n}g_{n-k}(x_2)x_0^{n-k},
\end{align*}
Comparing coefficients on both sides, we see,
\begin{align*}
\frac{x_2^m}{td}g_k(\frac{a}{x_2})a^k
=g_{n-k}(x_2), \qquad k=0,1,2\ldots n.
\end{align*}
Taking $k=0$ in above equation, and noticing that $t=g_{n}(0)a^{n-1}$, we have $n=1$ or $a=\pm1$. We thus obtain the period 1 polynomials in items (8), (9) and (10) of Theorem \ref{n3thm}.
Their generated seeds are
\begin{eqnarray*}
\{x_0, \pm x_1x_2+a\}, \{x_1, \pm x_0+x_2\}, \{x_2,\pm x_0x_1+a\}
\end{eqnarray*}
and
\begin{align*}
&\{x_0, 1+x_1^mx_2^n+\sum_{0<i<m\atop 0<j<n}C_{i,j}(x_1^ix_2^j+x_1^{m-i}x_2^{n-j})\},\\
&\{x_1, x_0^n+x_2^m+\sum_{0<i<m\atop 0<j<n}C_{i,j}(x_0^j x_2^{m-i}+x_0^{n-j}x_2^{i})\},\\
&\{x_2, 1+x_0^mx_1^n+\sum_{0<i<m\atop 0<j<n}C_{i,j}(x_0^ix_1^j+x_0^{m-i}x_1^{n-j})\}.
\end{align*}
When $m\equiv n\mod 2$,
\begin{align*}
&\{x_0, -1+(-1)^{m+1}x_1^mx_2^n+\sum_{0<i<m\atop 0<j<n}C_{i,j}(x_1^ix_2^j+(-1)^{m+j+i}x_1^{m-i}x_2^{n-j})\},\\
&\{x_1, -x_0^n-x_2^m+\sum_{0<i<m\atop 0<j<n}C_{i,j}((-1)^jx_0^j x_2^{m-i}+(-1)^{i}x_0^{n-j}x_2^{i})\},\\
&\{x_2, -1+(-1)^{m+1}x_0^mx_1^n+\sum_{0<i<m\atop 0<j<n}C_{i,j}(x_0^ix_1^j+(-1)^{m+j+i}x_0^{m-i}x_1^{n-j})\}.
\end{align*}

\end{enumerate}

\section{Examples of period 1 polynomials and seeds} \label{zoo}

In this section, we prove that several families of polynomials $P = P(x_1, \ldots, x_{n-1})$ are 1 periodic. In our first subsection, we prove the Expansion and Reflection Lemmas (Lemmas \ref{explemma} and \ref{reflemma}), which can be applied to period 1 polynomials to generate more period 1 polynomials. In the second subsection, we prove Theorem \ref{families}.

The proof that $P$ is a period 1 polynomial for each item in Theorem \ref{families} will simply consist of writing down the intermediate polynomials $P_i$. In general, it is easy to verify that $P_{i-1}=\tau_P(P_i)$ for all $i$ (and $P_{n-1} = \tau_P(P_0)$), showing the seed is a period 1 seed.

The importance of period 1 polynomials stems from Theorem \ref{period1} that says that if $\widehat{P}_0 = P_0$, then $P$ generates a Laurent phenomenon sequence. Theorem \ref{hat} gives sufficient conditions for $\widehat{P}_0 = P_0$ to be satisfied. In most of the seeds given below, the reader can easily verify that the intermediate polynomials $P_i$ that depend on $x_0$ are the ones for which $P = P_0$ depends on $x_i$, and so condition (1) of Theorem \ref{hat} is satisfied. The only exceptions will be the families in Subsections \ref{zooex9} and \ref{zooex10}, but these families satisfy condition (2) of Theorem \ref{hat} instead. Hence, the truth of Corollary \ref{laurentproperty} will follow from the seeds for the polynomials in Theorem \ref{families} that we give below.

\subsection{Proofs of the Expansion and Reflection Lemmas}

\subsubsection{Proof of Lemma \ref{explemma}}

Let $t = \bf (x,F)$ be the period 1 seed generated by $F$ and let $k\in\N$ be any positive integer. We prove that $G(x_1, x_2, \ldots, x_{nk-1}) = F(x_k, x_{2k}, \ldots, x_{(n-1)k})$ generates a period 1 seed.
Let $\textbf{G} = (G_1, \ldots, G_{kn})$, where $G_{k(i-1)+j}(x_1, \ldots, \widehat{x}_{k(i-1)+j}, \ldots, x_{kn}) = F_i(x_j, \ldots, \widehat{x}_{k(i-1) + j}, \ldots, x_{k(n-1) + j})$ for all $1\leq i\leq n$, $1\leq j\leq k$.
It is clear that $t' = \bf (y, G)$ is a seed and $G_{kn} = G$. It will then suffice to show that $t' = \bf (y, G)$ has period 1.
Observe that $G_1(x_2, \ldots, x_{kn}) = F_1(x_{k+1}, \ldots,  x_{k(n-1)+1})$ and $G_{kn}(x_1, \ldots, x_{kn-1}) = F_n(x_k, \ldots, x_{k(n-1)})$, so $G_{kn}$ is the downshift of $G_1$. We need to verify $G_s = \tau_{x_1, G_1}(G_{s+1})$ for all $1\leq s\leq nk-1$.

If $s \neq 0 \pmod{k}$, then $s = k(i-1) + j$ for some $i$ and $1 \leq j < k$. In this case, observe that $G_s$ is the downshift of $G_{s+1}$ by definition of $\textbf{G}$.
Moreover, the polynomial $G_{s+1}$ (and also ) only depends on the variables
$x_{j+1}, \ldots, \widehat{x}_{k(i-1)+j+1}, \ldots, x_{k(n-1)+j+1}$, and in particular, not $x_1$.
Hence, $\tau(G_{s+1})$ is the downshift of $G_{s+1}$, which is $G_s$ as remarked above.

If $s = 0 \pmod{k}$, then $s = k(i-1) + k$ for some positive integer $i$, and so $s+1 = ki + 1$. In this case, $G_s = F_i(x_k, \ldots, \widehat{x}_s, \ldots, x_{k(n-1)+1})$ and $G_{s+1} = F_{i+1}(x_1, \ldots, \widehat{x}_{s+1}, \ldots, x_{k(n-1) + 1})$. Since $(\bf{x}, \bf{F})$ is a period 1 seed, we have that $\tau_{x_1, F_1}(F_i) = F_{i-1}$.
Hence, $\tau_{x_1, G_1}(G_{s+1}) = F_{i-1}(x_k, \ldots, \widehat{x}_s, \ldots, x_n) = G_s$.

\subsubsection{Proof of Lemma \ref{reflemma}}
Let $\bf (x, F)$ be the period 1 seed whose intermediate polynomials are $F_i$, $0\leq i\leq n-1$. Define $G_i(x_0,\ldots, \widehat{x}_{i},\ldots,x_{n-1}) = F_{n-i-1}(x_{n-1},\ldots,\widehat{x}_{i},\ldots,x_1)$ for all $i$, and $\textbf{G} = (G_1, \ldots, G_{n})$.
We show that $\bf (x, G)$ is also a period 1 seed generated by $G$.

From the relation $\kappa_F(F_{n-i}) = F_{n-i+1}$, we have that replacing $x_n$ with $\displaystyle\frac{F_{n-1}|_{x_{n-i} = 0}}{x_{-1}}$ in $F_{n-i}$ and upshifting yields
\begin{align}
\label{eq1}F_{n-i} \left(x_1,\ldots,\widehat{x}_{n-i+1},\ldots,x_{n-1}, \frac{F_{n-1}(x_1,\ldots,x_{n-1})|_{x_{n-i+1}=0}}{x_0} \right).
\end{align}
Then, $F_{n-i+1}$ comes from dividing (\ref{eq1}) by the largest power of $F_{n-1}(x_1, \ldots, x_{n-1})|_{x_{n-i+1}=0}$ that divides it, and adjusting by a monomial factor.

We show that $\bf (x, G)$ is a period 1 seed by verifying that $\tau_G(G_i) = G_{i-1}$ for all $i$.
Similar to before, replacing $x_0$ with $\displaystyle\frac{G_0|_{x_i=0}}{x_n}$ in $G_i$ and downshifting yields
\begin{align}
G_i\big(\frac{G_0(x_0,x_1,\ldots, x_{n-2})|_{x_{i-1}=0}}{x_{n-1}}, x_0,x_1,\ldots,\widehat{x}_{i-1},\ldots, x_{n-2}\big)\nonumber\\
\label{eq2}= F_{n-i-1}(x_{n-2},\ldots,\widehat{x}_{i-1},\ldots,x_1,x_0,
\frac{F_{n-1}(x_{n-2},\ldots, x_0)|_{x_{i-1}=0}}{x_{n-1}}).
\end{align}
Then, $\tau_G(G_i)$ comes from dividing (\ref{eq2}) by the largest power of $G_0(x_0, \ldots, x_{n-2})|_{x_{i-1}=0}$
that divides it, and adjusting by a monomial factor.

Notice that replacing $x_j$ by $x_{n-j-1}$ for all $j$ in (\ref{eq1}) gives (\ref{eq2}).
Therefore,
\begin{align*}
\tau_G(G_i) = F_{n-i+1}(x_{n-1}, \ldots, \widehat{x}_{i-1}, \ldots, x_1) = G_{i-1}(x_0, \ldots, \widehat{x}_{i-1}, \ldots, x_{n-1}) = G_{i-1}.\end{align*}

\subsection{Period 1 polynomials and their generated seeds}

As remarked at the beginning of the section, we will show the seeds generated by the polynomials $P$ in Theorem \ref{families}.
In all cases, we obviously have $P_0 = P$ and $P_{n-1}$ be the downshift of $P$, so it will suffice to show the intermediate polynomials $P_i$ for $0 < i < n -1$.

\subsubsection{Gale-Robinson seed.}\label{zooex0}
We begin with the {\it Gale-Robinson polynomial} $P = Ax_px_{n-p} + Bx_qx_{n-q} + Cx_rx_{n-r}$, $p < q < r$ and $p + q + r = n$. There are many cases to consider when writing the seed for $P$. As the Laurent property for this polynomial is already well-known, we only write the seed in the case that $r < n/2$ (the other cases $q < n/2 < r$, $p < n/2 < r$, $n/2 < p$ and where there are some equalities among some of these quantities, are similar).

We first give the intermediate polynomials $P_i$ for $i\in\{p, q, r, n-p, n-q, n-r\}$:
\begin{itemize}
\item $P_p = ABx_qx_{2p}x_{p+r} + ACx_rx_{2p}x_{p+q} + Cx_0x_{p+r}x_{n+p-r} + Bx_0x_{p+q}x_{n+p-q}$.
\item $P_q = ABx_{q-p}x_px_{2q}x_{q+r} + ABx_0x_{2q-p}x_{p+q}x_{q+r} + ACx_0x_{r+q-p}x_{p+q}x_{2q} +$\\
$BCx_{q-p}x_rx_{p+q}x_{2q} + Cx_0x_{q-p}x_{q+r}x_{n+q-r}$.
\item $P_r = ABx_0x_{r-p}x_{p+r-q}x_{q+r}x_{2r} + ACx_{r-p}x_px_{r-q}x_{q+r}x_{2r} + ABx_0x_{q+r-p}x_{r-q}x_{p+r}x_{2r} +$\\
$BCx_{r-p}x_qx_{r-q}x_{p+r}x_{2r} + ACx_0x_{r-q}x_{2r-p}x_{p+r}x_{q+r} + BCx_0x_{r-p}x_{2r-q}x_{p+r}x_{q+r}$.
\item $P_{n-r} = ABx_qx_{p+q-r}x_{2p}x_{n+q-r} + ACx_qx_{2p+q-r}x_px_{n+q-r} + Cx_0x_{p+q-r}x_{n+p-r}x_{n+q-r} +$\\
$ABx_{p+q-r}x_{2q}x_px_{n+p-r} + BCx_qx_{n+q-2r}x_px_{n+p-r}$.
\item $P_{n-q} = ABx_rx_px_{n+p-2q} + ACx_rx_{2p}x_{p+r-q} + Bx_0x_{p+r-q}x_{n+p-q} + Cx_px_{r-q}x_{n+p-q}$.
\item $P_{n-p} = Ax_0x_{n-2p} + Bx_{q-p}x_r + Cx_{r-p}x_q$.
\end{itemize}
For the remaining polynomials $P_j$, pick the largest $i < j$ in the set $\{0, p, q, r, n-p, n-q, n-r\}$, and let $P_j$ be $P_i$ after upshifting $j-i$ times.

\subsubsection{Symmetric with second powers seed.}\label{zooex1}
If $P$ is of the form (1) in Theorem \ref{families}, the intermediate polynomials are
$\displaystyle P_i = P(x_0, \ldots, \widehat{x}_i, \ldots, x_{n-1})$ for all $0 < i < n-1$.

\begin{exam}\label{zooexample1}
When $n=3$, these family of polynomials accounts for the family (6) in Theorem \ref{n3thm}.
\end{exam}

\subsubsection{Sink-type binomial seed.}\label{zooex2}
If $P$ is of the form (2) in Theorem \ref{families}, the intermediate polynomials are
$\displaystyle P_i = \prod_{j=0}^{i-1}{x_j^{a_{n-i-j}}} + \prod_{j=1}^{n-i-1}{x_{i+j}^{a_j}}$ for all $0 < i < n-1$.

\begin{exam}\label{zooexample2}
When $n=6$, the polynomial $P = x_1^2x_3^3x_5 + 1$ generates the period 1 seed
\begin{eqnarray*}
\{x_0, x_1^2x_3^3x_5 + 1\}, \{x_1, x_2^2x_4^3 + x_0\}, \{x_2, x_3^2x_5^3 + x_1\}, \{x_3, x_0^3x_2 + x_4^2\}, \{x_4, x_1^3x_3 + x_5^2\}, \{x_5, x_0^2x_2^3x_4 + 1\}.
\end{eqnarray*}
\end{exam}

\subsubsection{Extreme seed.}\label{zooex3}
If $P$ is of the form (3) in Theorem \ref{families}, the intermediate polynomials are
$P_i = x_{i-1} + x_{i+1} + A$ for all $0 < i < n-1$.

\begin{exam}\label{zooexample3}
When $n=4$, the polynomial $P = x_1x_3 + 3(x_1 + x_2 + x_3) + 2$ generates the period 1 seed
\begin{align*}
\{x_0, x_1x_3 + 3(x_1 + x_2 + x_3) + 2\}, \{x_1, x_0 + x_2 + 3\},
\{x_2, x_1 + x_3 + 3\}, \{x_3, x_0x_2 + 3(x_0 + x_1 + x_2) + 2\}.
\end{align*}
\end{exam}

\subsubsection{Singleton seed.}\label{zooex4}
If $P$ is a single variable polynomial of the form (4) in Theorem \ref{families}, the intermediate polynomials are $\displaystyle P_i = P\left(x_{i + \frac{n}{2}\pmod{n}}\right)$.

\begin{exam}\label{zooexample4}
When $n=4$, the polynomial $P = x_2^2 + 2x_2 - 7$ generates the period 1 seed
\begin{eqnarray*}
\{x_0, x_2^2 + 2x_2 - 7\}, \{x_1, x_3^2 + 2x_3 - 7\}, \{x_2, x_0^2 + 2x_0 - 7\}, \{x_3, x_1^2 + 2x_1 - 7\}, \{x_4, x_3^2 + 2x_3 - 7\}.
\end{eqnarray*}
\end{exam}

\begin{rem}
These polynomials correspond to $(\frac{n}{2})$-expansions of the period 1 polynomials found in Theorem \ref{n2thm}.
\end{rem}

\subsubsection{Chain seed.}\label{zooex5}
If $P$ is of the form (5) in Theorem \ref{families}, the intermediate polynomials are
\begin{eqnarray*}
P_{2j-1} &=& x_0 + x_{n-1} + A \textrm{, for all } 0 < j \leq (n-1)/2\\
P_{2j} &=& F(x_0, \ldots, \widehat{x}_{2j}, \ldots, x_{n-1}) \textrm{, for all } 0 < j < (n-1)/2.
\end{eqnarray*}

\subsubsection{Multilinear symmetric seed.}\label{zooex6}
If $P$ is of the form (6) in Theorem \ref{families}, the intermediate polynomials are
\begin{eqnarray*}
P_{2j-1} &=& E_1(x_0, \ldots, \widehat{x}_{2j-1}, \ldots, x_{n-1}) + A \textrm{, for all } 0 < j \leq (n-1)/2\\
P_{2j} &=& F(x_0, \ldots, \widehat{x}_{2j}, \ldots, x_{n-1}) \textrm{, for all } 0 < j < (n-1)/2.
\end{eqnarray*}

\subsubsection{$r$-Jumping seed.}\label{zooex7}
Let $r, n, A\in\N$ be constants and $P$ a polynomial in the setup of (7) of Theorem \ref{families}.
For any $a\geq 0$ such that $a + r < n$, define \[ F_a =
\sum_{k = 0}^{\lfloor \frac{n - a}{r} \rfloor - 1}{x_{a + rk}\cdot x_{a + rk + r - 1}} + A. \]
With this definition, notice that $P = F_1$.
The intermediate polynomials are:
\begin{itemize}
	\item $\displaystyle P_j = \big(\sum_{i=1}^{j+1}{a_i^j\cdot F_{j+2-i}}\big)\big|_{x_j = 0}$ and
	$\displaystyle P_{n - j - 1} = \big(b_0^j\cdot F_0 + \sum_{i=1}^j{b_i^j\cdot F_{r-i}}\big)\big|_{x_{n-j-1} = 0}$ for all $1\leq j\leq r - 2$, where $\displaystyle a_i^j = \prod_{k=0}^{j-i}{x_k}\cdot\prod_{k=0}^{i-2}{x_{n-r+j-k}}$ and
	$\displaystyle b_i^j = \prod_{k=0}^{j-i-1}{x_{r-j+k}}\cdot\prod_{k=1}^i{x_{n-k}}$.
	\item $\displaystyle P_j = (\sum_{i=1}^r{a_i\cdot F_{r-i+1}})\big|_{x_j = 0}$ for all $r-1\leq j\leq n-r$, where $\displaystyle a_i = \prod_{k=0}^{r-1-i}{x_k}\cdot\prod_{k=0}^{i-2}{x_{n-k-1}}$.
\end{itemize}

\begin{exam}\label{zooexample7}
When $n=7$, the polynomial $P = x_1x_3 + x_4x_6$ generates the period 1 seed
\begin{align*}
\{x_0,\quad& x_1x_3 + x_4x_6\},\\
\{x_1,\quad& x_0x_2x_4 + x_4x_5x_6\},\\
\{x_2,\quad& x_0x_1x_3x_5 + x_1x_3x_5x_6 + x_4x_5x_6^2\},\\
\{x_3,\quad& x_0^2x_1x_2 + x_0x_2x_4x_6 + x_4x_5x_6^2\},\\
\{x_4,\quad& x_0^2x_1x_2 + x_0x_1x_3x_5 + x_1x_3x_5x_6\},\\
\{x_5,\quad& x_0x_1x_2 + x_2x_4x_6\},\\
\{x_6,\quad& x_0x_2 + x_3x_5\}
\end{align*}
Observe that this is an example of a binomial that generates a period 1 seed whose intermediate polynomials are not all binomials. Jumping polynomials are not classified by Theorem \ref{binomialthm}.
\end{exam}

\subsubsection{$r$-Hopping seed.}\label{zooex8}
Let $r, n, A, B\in\N$ be constants and $P$ a polynomial in the setup of (8) of Theorem \ref{families}.
For any $0\leq a < n-r$, define \[ F_a = \sum_{k=0}^{\lfloor \frac{n - a}{r} \rfloor - 1}{x_{a + rk}\cdot x_{a + rk + r - 1}} + A\sum_{k=0}^{\lfloor\frac{n-a}{r}\rfloor - 2}{x_{a + rk + r - 1}\cdot x_{a + rk + r}} + B. \]
With this definition, notice that $P = F_1$.
The intermediate polynomials are:

\begin{itemize}
	\item $\displaystyle P_j = \big(\sum_{i=1}^{j+1}{a_i^j\cdot F_{j+2-i}}\big)\big|_{x_j = 0}$ and
	$\displaystyle P_{n-j-1} = \big(b_0^j\cdot F_0 + \sum_{i=1}^j{b_i^j\cdot F_{r-i}}\big)\big|_{x_{n-j-1} = 0}$ for all $1\leq j\leq r - 2$, where
$\displaystyle a_i^j = \prod_{k=0}^{j-i}{x_k}\cdot\prod_{k=0}^{i-2}{x_{n-r+j-k}}$ and $\displaystyle b_i^j = \prod_{k=0}^{j-i-1}{x_{r-j+k}}\cdot\prod_{k=1}^i{x_{n-k}}$.
	\item $\displaystyle P_j = (\sum_{i=1}^r{a_i\cdot F_{r-i+1}})\big|_{x_j = 0}$ for all $r-1\leq j\leq n-r$, where $\displaystyle a_i = \prod_{k=0}^{r-1-i}{x_k}\cdot\prod_{k=0}^{i-2}{x_{n-k-1}}$.
\end{itemize}

\begin{rem}
In the definitions of $a_i^j, b_i^j$ and $a_i$ in the jumping and hopping seeds, a product $\displaystyle\prod_{k=L}^M{X_k}$ is defined to be $1$ if $M < L$.
\end{rem}

\subsubsection{Flip-symmetric binomial seed.}\label{zooex9}
These are the seeds discussed on Section 4. We give an explicit description here for consistency.

Let $L, R \subset [n-1]$ be disjoint subsets, $a: L\cup R\rightarrow\N$ a map and $P$ a polynomial in the setup of (9) of Theorem \ref{families}.

Let $\bf a$ = $(a_0, a_1, \ldots, a_{n-1}) \in \Z_{\geq 0}^n$ be the vector with nonnegative entries such that $a_i = a(i+1)$ if $i+1\in L$, $a_j = -a(j+1)$ if $j+1\in R$ and $a_k = 0$ for the remaining indices $k$. Define the vectors $\bf b^{(1)}, \ldots, b^{(n-1)}$ recursively as follows. Let $\bf b^{(i)}$ = $(b_0^{(i)}, \ldots, b_{n-1}^{(i)})$, $\bf c^{(i)}$ = $(c_0^{(i)}, \ldots, c_{n-1}^{(i)})$ for all $i$ and begin defining $\bf b^{(1)} = a$. Then let $c_0^{(i)} = -b_0^{(i-1)}$ and $c_j^{(i)} = b_j^{(i-1)} + b_0^{(i-1)} \cdot |a_j| \cdot \mathbf{1}_{\{a_{n-i}a_j < 0\}}$ for all $0\leq j\leq n-1$ and $i>1$ (the indicator function $\mathbf{1}_{\{a_{n-i}a_j < 0\}}$ is $1$ if $a_{n-i}a_j < 0$ and is $0$ otherwise). Finally, let $\bf b^{(i)}$ be the vector that comes from permuting $\bf c^{(i)}$ with the permutation $(0, 1, 2, \ldots, n-1) \rightarrow (1, 2, \ldots, n-1, 0)$.
We can now show the intermediate polynomials $P_i$ for $0 < i < n-1$. Polynomial $P_i$ is derived from vector $\bf b^{(n-i)}$ as follows:
let $L_i$ (resp. $R_i$) be the set of indices $0\leq k\leq n-1$ such that $b_k^{(n-i)} > 0$ (resp. $b_k^{(n-i)} < 0$). Then $\displaystyle P_i = \prod_{k\in L_i}{x_i^{b_k^{(n-i)}}} + \prod_{k\in R_i}{x_i^{-b_k^{(n-i)}}}$.

From Theorem \ref{mutualdoublequiver} and the definition of B-matrix mutation in (\ref{bmatrixmutation}), the resulting seed $\bf (x, P)$ is a period 1 seed.

\begin{exam}\label{zooexample9}
When $n = 8$, the polynomial $P = x_1^3x_7^2 + x_4^3x_2x_6$ generates the period 1 seed
\begin{align*}
\{x_0, x_1^3x_7 + x_2x_4^3x_6\},
\{x_1, x_0^2x_3x_5^3x_7 + x_2^5x_4^6x_6^2\},
\{x_2, x_0x_3^5x_5^6 + x_1^5x_4x_6^3\},
\{x_3, x_1x_4^5x_6^6 + x_2^5x_5x_7^3\},\\
\{x_4, x_0^3x_2x_5^5 + x_1^9x_3^5x_6\},
\{x_5, x_1^3x_3x_6^5 + x_2^9x_4^5x_7\},
\{x_6, x_0x_2^3x_4x_7^3 + x_1^3x_3^9x_5^5\},
\{x_7, x_0^3x_6 + x_1x_3^3x_5\}
\end{align*}
We now demonstrate how to obtain $\bf b^{(5)}$ from $\bf b^{(4)}$. We have $\bf a$ = $(0, 3, -1, 0, -3, 0, -1, 2)$ (corresponding to polynomial $P$) and $\bf b^{(4)}$ = $(3, -9, 1, -5, 0, 5, -1, 0)$ (corresponding to polynomial $P_4$).

Notice that $a_5 = -3$ implies $a_5a_j < 0$ if and only if $a_j > 0$.
The only indices $j$ for which $a_j > 0$ are $1$ and $7$.
Then $c_1^{(0)} = -b_1^{(4)}$, $c_1^{(5)} = b_1^{(4)} + b_0^{(4)} \cdot |a_1| = -9 + (3)(3) = 0$, $c_7^{(5)} = b_7^{(4)} + b_0^{(4)} \cdot |a_7| = 0 + (3)(2) = 6$ and $c_j^{(5)} = b_j^{(4)}$ for the remaining indices $j$. Thus $\bf c^{(5)}$ = $(-3, 0, 1, -5, 0, 5, -1, 6)$ and $\bf b^{(5)}$ = $(0, 1, -5, 0, 5, -1, 6, -3)$. This corresponds to $P_3 = x_1x_4^5x_6^6 + x_2^5x_5x_7^3$.
\end{exam}

\subsubsection{Balanced seed.}\label{zooex10}
Let $L, R \subset [n-1]$ be disjoint subsets, $a: L\cup R\rightarrow\N$ a map, $M_1, M_2$ monomials, $m>1$ an integer and $P$ a polynomial in the setup of (10) of Theorem \ref{families}.

Notice that $P' = M_1^m + M_2^m$ is a binomial with flip-symmetry seed, so it generates a period 1 seed.
Let $P_j'$ be the intermediate polynomials of this seed.
From the analysis in \ref{zooex9}, we see that each $P_i'$ is a binomial of the form $M_1^{(i)} + M_2^{(i)}$, where each of the monomials $M_1^{(i)}, M_2^{(i)}$ is an $m$-th power of a monomial. Say that $(N_1^{(i)})^m = M_1^{(i)}$ and $(N_2^{(i)})^m = M_2^{(i)}$, then the intermediate polynomials are $\displaystyle P_i = (N_1^{(i)})^m + (N_2^{(i)})^m + \sum_{k=1}^{\lfloor m/2 \rfloor}{(A_i\cdot (N_1^{(i)})^k (N_2^{(i)})^{m-k} + A_i\cdot (N_1^{(i)})^{m-k}(N_2^{(i)})^k)}$.

\subsubsection{Vector sum seed.}\label{zooex11}
Let $a_1, \ldots, a_{n-1}\in\N$ be constants, $B$ a finite set of vectors $\N^{n-1}$ and $P$ a polynomial in the setup of (11) in Theorem \ref{families}. For ease of notation, let $B' = B\cup\{(0, \ldots, 0)\}$ and $C_{(0, \ldots, 0)} = 1$. The intermediate polynomials are
\begin{eqnarray*}
P_i &=& \sum_{b\in B'}{(C_b\cdot x_0^{b_{n-i}}\ldots x_{i-1}^{b_{n-1}}x_{i+1}^{a_1 - b_1}\ldots x_{n-1}^{a_{n-1-i} - b_{n-1-i}})}\\
&+& \sum_{b\in B'}{(C_b\cdot x_0^{a_{n-i} - b_{b_{n-i}}}\ldots x_{i-1}^{a_{n-1}-b_{n-1}}x_{i+1}^{b_1}\ldots x_{n-1}^{b_{n-1-i}})} \textrm{, for all $0 < i < n-1$}.
\end{eqnarray*}

\begin{exam}\label{zooexample11}
The polynomial $P = 1 + x_1^3x_2^2x_3^4x_4^2 + 2x_1x_2x_3^2x_4 + 2x_1^2x_2x_3^2x_4$ generates the period 1 seed
\begin{align*}
\{x_0,\quad& 1 + x_1^3x_2^2x_3^4x_4^2 + 2x_1x_2x_3^2x_4 + 2x_1^2x_2x_3^2x_4\},\\
\{x_1,\quad& x_0^2 + x_2^3x_3^2x_4^4 + 2x_0x_2^2x_3x_4^2 + 2x_0x_2x_3x_4^2\},\\
\{x_2,\quad& x_3^3x_4^2 + x_0^4x_1^2 + 2x_0^2x_1x_3^2x_4 + 2x_0^2x_1x_3x_4\},\\
\{x_3,\quad& x_4^3 + x_0^2x_1^4x_2^2 + 2x_0x_1^2x_2x_4^2 + 2x_0x_1^2x_2x_4
\},\\
\{x_4,\quad& 1 + x_0^3x_1^2x_2^4x_3^2 + 2x_0x_1x_2^2x_3 + 2x_0^2x_1x_2^2x_3\}
\end{align*}
\end{exam}

\subsubsection{Little Pi Seed.}\label{zooex12}
Let $k, n\in\N$ be constants and $P$ a polynomial in the setup of (12) of Theorem \ref{families}. We show the intermediate polynomials $P_j$ for $j \in\{k, 2k, n-2k, n-k\}$. For the general $P_i$, if $j$ is the largest integer with $j < i$ and $j\in J$, then $P_i$ comes from $i - j$ upshifts to $P_j$.

{\bf Case 1:} If $n > 4k$, so that $k < 2k < n-2k < n-k$, then
\begin{itemize}
\item $P_k = Ax_0x_{2k} + Ax_{2k}x_{n-2k} + x_0x_{3k}x_{n-k} + A^2x_{n-k}$
\item $P_{2k} = x_0x_{3k} + x_k x_{4k} + A^2$
\item $P_{n-2k} = Ax_kx_{n-3k} + Ax_{n-3k}x_{n-k} + x_0x_{n-4k}x_{n-k} + A^2x_0$
\item $P_{n-k} = Ax_0 + Ax_{n-2k} + x_kx_{n-3k}.$
\end{itemize}

{\bf Case 2:} If $n = 4k$, so that $k < 2k = n-2k = 2k < n-k = 3k$, then
\begin{itemize}
\item $P_k = Ax_0x_{2k} + Ax_{2k}^2 + x_0x_{3k}^2 + A^2x_{3k}$
\item $P_{2k} = Ax_k^2 + Ax_kx_{3k} + x_0^2x_{3k} + A^2x_0$
\item $P_{3k} = Ax_0 + Ax_{2k} + x_k^2.$
\end{itemize}

{\bf Case 3:} If $4k > n > 3k$, so that $k < n-2k < 2k < n-k$, then
\begin{itemize}
\item $P_k = Ax_0x_{2k} + Ax_{2k}x_{n-2k} + x_0x_{3k}x_{n-k} + A^2x_{n-k}$
\item $P_{n-2k} = x_0x_{n-3k}x_{n-k} + x_0x_{2n-5k}x_{n-k} + x_{n-3k}x_kx_{2n-4k} + x_{n-3k}x_{n-k}x_{2n-4k} + Ax_0x_{2n-4k}$
\item $P_{2k} = Ax_kx_{5k-n} + Ax_kx_{3k} + x_0x_kx_{3k} + A^2x_{4k-n}$
\item $P_{n-k} = Ax_0 + Ax_{n-2k} + x_kx_{n-3k}.$
\end{itemize}

{\bf Case 4:} If $3k > n > 2k$, so that $n-2k < k < n-k < 2k$, then
\begin{itemize}
\item $P_{n-2k} = x_{2n-4k}x_k + x_{2n-4k}x_{n-k} + x_0x_{2k} + x_0x_{2n-3k}$
\item $P_k = x_0x_{n-k}x_{4k-n} + x_0x_{n-k}x_{2k} + x_0x_{3k-n}x_{2k} +    x_{n-2k}x_{3k-n}x_{2k} + Ax_{3k-n}x_{n-k}$
\item $P_{n-k} = x_{n-2k}x_k + x_{2n-3k}x_{n-2k} + x_0x_{2n-3k} + x_kx_{2n-4k}$
\item $P_{2k} = Ax_{3k-n} + Ax_{k} + x_0x_{4k-n}.$
\end{itemize}

\subsection{Pi Seed.}\label{zooex13}
Let $k, n, a, b\in\N$ be constants and $P$ a polynomial in the setup of (13) in Theorem \ref{families}. We show the intermediate polynomials $P_j$ for $j\in\{k, 2k, n-2k, n-k\}$. We obtain the remaining intermediate polynomials $P_i$ as before. Without loss of generality, assume $a_2\geq a_1$ and $b_1\geq b_2$.

{\bf Case 1:} If $n > 4k$, so that $k < 2k < n-2k < n-k$, then
\begin{itemize}
    \item $P_k = Ax_{2k}^{a_2+b_2}x_{n-2k}^{b_1} + Bx_{2k}^{a_2+b_2}x_{n-2k}^{b_2}x_0^{b_1-b_2} + x_0^{b_1}x_{3k}x_{n-k}$.
    \item $P_{2k} = x_0x_{3k}^{a_2+b_2} + x_k^{a_1+b_1}x_{4k}$.
    \item $P_{n-2k} = Ax_k^{a_1}x_{n-3k}^{a_1+b_1}x_{n-k}^{a_2-a_1} + Bx_k^{a_2}x_{n-3k}^{a_2+b_2}+ x_0x_{n-4k}x_{n-k}^{a_2}$.
    \item $P_{n-k} = Ax_0^{a_1}x_{n-2k}^{b_1} + Bx_0^{a_2}x_{n-2k}^{b_2} + x_kx_{n-3k}$.
\end{itemize}

{\bf Case 2:} If $n = 4k$, so that $k < 2k = n-2k < n-k = 3k$, then

\begin{itemize}
    \item $P_k = Ax_{2k}^{a_2+b_1+b_2}+ Bx_{2k}^{a_2+2b_2}x_0^{a_2-a_1} + x_0^{b_1}x_{3k}^2$.
    \item $P_{2k} = Ax_k^{2a_1+b_1}x_{3k}^{a_2-a_1} + Bx_k^{2a_2 + b_2} + x_0^2x_{3k}^{a_2}$.
    \item $P_{3k} = Ax_0^{a_1}x_{2k}^{b_1} + Bx_0^{a_2}x_{2k}^{b_2} + x_k^2$.
\end{itemize}

{\bf Case 3:} If $4k > n > 3k$, so that $k < n-2k < 2k < n-k$, then

\begin{itemize}
    \item $P_k = Ax_{2k}^{a_2+b_2}x_{n-2k}^{b_1} + Bx_{0}^{a_2-a_1}x_{2k}^{a_2+b_2}x_{n-2k}^{b_2} + x_0^{b_1}x_{3k}x_{n-k}$.
    \item $P_{n-2k} = Ax_0x_{n-3k}^{a_1}x_{2n-5k}^{b_1}x_{n-k}^{a_2} + Bx_0x_{2n-5k}^{b_2}x_{n-3k}^{a_2}x_{n-k}^{a_2} + Ax_{n-3k}^{b_1}x_k^{a_1}x_{2n-4k}x_{n-k}^{a_2-a_1}+ Bx_{n-3k}^{b_1}x_k^{a_2}x_{2n-4k}$.
	\item $P_{2k} = Ax_{3k}^{a_2-a_1}x_{5k-n}^{a_1}x_k^{a_1+b_1} + Bx_k^{a_2+b_2}x_{5k-n}^{a_2} + x_0x_{4k-n}x_{3k}^{a_2}$.
    \item $P_{n-k} = Ax_0^{a_1}x_{n-2k}^{b_1} + Bx_0^{a_2}x_{n-2k}^{b_2} + x_kx_{n-3k}$.
\end{itemize}

{\bf Case 4:} If $3k > n > 2k$, so that $n-2k < k < n-k < 2k$. Two cases will arise; for simplicity, let us only do the case $a_1 \leq b_2$.
\begin{itemize}
    \item $P_{n-2k} = Ax_{2n-4k}x_k^{a_1}x_{n-k}^{b_1-a_1} + Bx_{2n-4k}x_k^{a_2}x_{n-k}^{b_1-a_2} + Ax_0x_{2n-3k}^{b_1} + Bx_0x_{n-k}^{a_2-a_1}x_{2n-3k}^{b_2}$.
	\item $P_k = Ax_0^{b_1}x_{2k}^{a_2-a_1}x_{4k-n}^{a_1}x_{n-k} + Bx_0^{b_1}x_{n-k}x_{4k-n}^{a_2} + Ax_{3k-n}x_{n-2k}^{b_1}x_{2k}^{a_2} + Bx_0^{a_2-a_1}x_{n-2k}^{b_2}x_{2k}^{a_2}x_{3k-n}$.
    \item $P_{n-k} = Ax_0^{a_1}x_{n-2k}^{b_1-a_1}x_{2n-3k} + Bx_0^{a_2}x_{2n-3k}x_{n-2k}^{b_2-a_1} + Ax_k x_{2n-4k}^{b_1} + Bx_{n-2k}^{a_2-a_1}x_{2n-4k}^{b_2}x_{k}$.
    \item $P_{2k} = Ax_{3k-n}^{a_1}x_k^{b_1} + Bx_{3k-n}^{a_2}x_k^{b_2} + x_0x_{4k-n}$.
\end{itemize}

\begin{exam}\label{zooexample13}
When $n=8$, $k=2$ (Case $n = 4k$), $a = 3$ and $b = 2$, the polynomial $P = -2x_2^3x_6^2 + 3x_2^2x_6^3 + x_4^2$ generates the period 1 seed
\begin{align*}
&\{x_0, -2x_2^3x_6^2 + 3x_2^2x_6^3 + x_4^2\},
\{x_1, -2x_3^3x_7^2 + 3x_3^2x_7^3 + x_5^2\},
\{x_2, -2x_0x_4^7 + 3x_4^8 + x_0^3x_6^2\},\\
&\{x_3, -2x_1x_5^7 + 3x_5^8 + x_1^3x_7^2\},
\{x_4, -2x_2^8 + 3x_2^7x_6 + x_0^2x_6^3\},
\{x_5, -2x_3^8 + 3x_3^7x_7 + x_1^2x_7^3\},\\
&\{x_6, -2x_0^3x_4^2 + 3x_0^2x_4^3 + x_2^2\},
\{x_7, -2x_1^3x_5^2 + 3x_1^2x_5^3 + x_3^2\}.
\end{align*}
\end{exam}

\section{Conserved quantities and k-invariants}\label{conserved}

In this section, we examine the integrablity of the sequences generated by some period 1 polynomials. Our general approach, for each sequence, is to find a {\it conserved quantity}, which we will denote by $J$, of the recurrence. A conserved quantity is a rational polynomial function depending on any $n$ consecutive terms of the sequence, i.e., $J_{m, n} = J(x_m, x_{m+1}, \ldots, x_{m+n-1})$ is independent of $m$; in other words $J_{m+1, n} = J_{m, n} = J$. Using this conserved quantity, we {\it multilinearize} the recurrence $\displaystyle x_{m+n} = P(x_{m+1}, \ldots, x_{m+n-1})/x_m$ by writing it in the equivalent form
\[x_{m+n}=L(x_m,x_{m+1},\ldots,x_{m+n-1}),\]
where $L$ is a multilinear polynomial with coefficients in $\Q[x_0,\ldots,x_{n-1}]$.
If $L$ is linear, we say that the recurrence has been {\it linearized}. Notice that when the coefficients of $L$ are all Laurent polynomials in $x_0,\ldots,x_{n-1}$, this multilinearization provides an alternate proof of the Laurent phenomenon for these sequences. In some cases, we find a {\it k-invariant} instead of a conserved quantity. This is a rational polynomial function $J$ such that $J_{m, n} = J(x_m, x_{m+1}, \ldots, x_{m+n-1})$ depends on the residue of $m$ modulo $k$; in other words $J_{m+k}=J_{m,n}$.

Since proving the integrablity of a sequence is technical and involves detailed discussion for each sequence, we will only provide an integrability test analysis for the first sequence discussed and leave the rest to the reader. For the relevant discussion on integrability of sequences, including cluster algebras and poisson geometry, refer to \cite{fo},\cite{ho2} and \cite{gsv}.

\subsection{Special case of symmetric with second powers polynomial}
We obtain a conserved quantity for the recurrence defined by \[ P = \sum_{i=1}^{n-1}{x_i^2} + A\sum_{i=1}^{n-1}{x_i} + B. \]
The recurrence at indices $m+n$ and $m+n+1$ are:
\begin{eqnarray*}
x_{m+n}x_{n} &=& \sum_{i=1}^{n-1}{x_{m+i}^2} + A\sum_{i=1}^{n-1}{x_{m+i}} + B\\
x_{m+n+1}x_{n+1} &=& \sum_{i=1}^{n-1}{x_{m+i+1}^2} + A\sum_{i=1}^{n-1}{x_{m+i+1}} + B
\end{eqnarray*}
After subtracting the former from the latter and rearranging, we obtain
\begin{align*}
\frac{x_{m+1} + x_{m+n+1} + A}{\prod_{i=1}^{n-1}{x_{m+i+1}}}=\frac{x_m + x_{m+n} + A}{\prod_{i=1}^{n-1}{x_{m+i}}}.
\end{align*}
Therefore $\displaystyle\frac{x_m + x_{m+n} + A}{\prod_{i=1}^{n-1}{x_{m+i}}}$ is a conserved quantity for our recurrence that we will write as $J_{m,n}$. By multiplying the numerator and denominator by $x_m$, and using $x_{m+n}x_m = A\sum_{i=1}^{n-1}{x_i} + B$ in the numerator, we see that $J_{m, n}$ can be written as \[ J_{m, n} = \frac{\sum_{i=0}^{n-1}{x_{m+i}^2} + A\cdot\sum_{i=0}^{n-1}{x_{m+i}} + B}{\prod_{i=0}^{n-1}{x_{m+i}}}. \] Then we have
$$ J_{m+1,n} = J_{m,n} = J = \frac{\sum_{i=0}^{n-1}{x_{i}^2} + A\cdot\sum_{i=0}^{n-1}{x_{i}} + B}{\prod_{i=0}^{n-1}{x_{i}}},$$
and the mutilinear recurrence
\begin{eqnarray}\label{linrec1}
x_{m+n}=J\cdot\prod_{i=1}^{n-1}x_{m+i}-x_{m}-A, \qquad m=0,1,2\ldots
\end{eqnarray}
We show that the recurrence generated by $P$ passes the singularity confinement test described in \cite{ho}. Assume we had a singularity at $x_{m+n}$, i.e., $x_{m+n} = \epsilon \rightarrow 0$. Then we have $\epsilon x_m = P(x_{m+1},x_{m+2},\ldots x_{m+n-1})=O(\epsilon)$. From (\ref{linrec1}), $x_{m+n}=J\cdot\prod_{i=1}^{n-1}x_{m+i}-x_m-A$, we can show inductively that $x_{n+m+i}=-x_{n+i}-A+O(\epsilon)$ for $i=1,\ldots, m-1$. It is therefore clear that $x_{n+2m}=O(1)$, that is, the singularity is confined. It is interesting to observe that, even though the sequence passes this singularity confinement test, it is not Diophantine integrable, as shown in \cite{ho}.

Another interesting fact is that the quadratic Diophantine equation
$$\sum_{i=0}^{n-1}{x_i^2} + A\cdot\sum_{i=0}^{n-1}{x_i} + B = \left(n(1+A) + B\right)\cdot\prod_{i=0}^{n-1}{x_i}$$
has infinitely many integer solutions $(x_0, x_1, \ldots, x_{n-1})$ that consist of the $n$-tuples $(y_m, y_{m+1}, \ldots, y_{m+n-1})$, where the sequence $(y_m)_{m=0}^{\infty}$ is defined as $y_0 = y_1 = \ldots y_{n-1} = 1$ and $y_{m+n} = F(y_{m+1}, \ldots, y_{m+n-1})/y_m$ for all $m\geq 0$.

\subsection{$r$-Jumping polynomial}
A conserved quantity for the $r$-Jumping polynomial
\[ P = \sum_{i = 0}^{\frac{n-1}{r} - 1}{x_{ri + 1}\cdot x_{ri + r}} + A, \] when $n \geq 2r+1$ and $n\equiv 1 \pmod{r}$, is
$$J_{m, n} = \frac{x_{m+1} + x_{m+n+r}}{\prod_{i=r+1}^n{x_{m+i}}}.$$
We then have
\begin{align*}
J_{m+1,n}=J_{m,n}=J=\frac{x_1+x_{n+r}}{\prod_{i=r+1}^{n}x_i},
\end{align*}
as well as the multilinear recurrence
\begin{align*}
x_{m+n+r}=J\cdot\prod_{i=r+1}^{n}x_{m+i}-x_{m+1},\qquad m=0,1,2\ldots
\end{align*}

\subsection{Special case of sink-type binomial}
This is the first example of polynomial for which we find a k-invariant instead of a conserved quantity. For the polynomial $\displaystyle P = x_kx_{n-k} + 1$,
where $0 < k < n$, there is a $(n-k)$-invariant which is
$$J_{m, n} = \frac{x_m + x_{m+2k}}{x_{m+k}}.$$
We then have
$$J_{m+n-k,n} = J_{m,n}.$$
The quantity $J_{m, n}$ will depend on the residue of $m$ modulo $n-k$; more specifically:
$$J_{m, n} = J_i = \frac{x_i + x_{i+2k}}{x_{i+k}} \textrm{ if } m\equiv i \pmod{n-k} \textrm{ and } 0 \leq i < n-k; $$
moreover, we obtain the linear recurrence
$$x_{m+2k}=J_{(m \textrm{ mod } n-k)}x_{m+k} - x_m,\qquad m = 0,1,2\ldots$$
This recurrence is thoroughly discussed in \cite{fm}, where it is shown that the sequence is given by its initial values and a recurrence $x_{m+n} = G(x_m, x_{m+1}, \ldots, x_{m+n-1})$ for a linear function $G$. Moreover, it is shown there that the sequence is {\it complete integrable}.

\subsection{Extreme polynomial}
A $(n-1)$-invariant for \[ P = x_1x_{n-1}+A\cdot\sum_{i=1}^{n-1} +B \] is
$$ J_{m,n}=\frac{x_{m+2}+x_{m}+A}{x_{m+1}}.$$
We then have
$$ J_{m+n-1,n}=J_{m,n}. $$
The quantity $J_{m, n}$ will depend on the residue of $m$ modulo $n-1$; more specifically:
$$J_{m, n} = J_{i}=\frac{x_{i+2}+x_i+A}{x_{i+1}} \textrm{ if } m\equiv i\pmod{n-1} \textrm{ and } 0 \leq i < n-1;$$
moreover, we obtain the linear recurrence
$$ x_{m+2} = J_{(m \textrm{ mod }n-1)}x_{m+1} - x_m - A,\qquad m= 0,1,2\ldots$$

\subsection{Chain polynomial}
A $2$-invariant for \[ P = \sum_{i=1}^{n-2}x_ix_{i+1} + A\cdot\sum_{i=1}^{n-1}x_i +B, \] when $n$ is odd, is
$$ J_{m,n}=\frac{x_{m+n-1}+x_{m}+A}{\prod_{i=0}^{\frac{n-3}{2}}x_{m+2i+1}}.$$
We then have
\begin{align*} J_{m+2,n}&=J_{m,n}=J_0=\frac{x_{n-1}+x_{0}+A}{\prod_{i=0}^{\frac{n-3}{2}}x_{2i+1}},\qquad 2\mid m\\
J_{m+2,n}&=J_{m,n}=
J_1=\frac{\sum_{i=1}^{n-2}x_ix_{i+1}+A\sum_{i=1}^{n-1}x_i+x_{1}+A+B}
{\prod_{i=0}^{\frac{n-3}{2}}x_{2i+2}},\qquad 2\nmid m
\end{align*}
and the mutilinear recurrence
\begin{align*}
x_{m+n-1}=J_{\langle m \text{ mod } 2\rangle}\cdot\prod_{i=0}^{\frac{n-3}{2}}x_{m+2i+1}-x_m-A,\qquad m=0,1,2\ldots
\end{align*}

\subsection{Multilinear symmetric polynomial}
A $2$-invariant for \[ P = \sum_{1\leq i<j\leq n-1}x_ix_{j}+A\sum_{i=1}^{n-1}x_i +B, \] when $n$ is odd, is
$$J_{m, n} = \frac{\sum_{i=0}^{n-1}{x_{m+i}} + A}{\prod_{i=0}^{\frac{n-3}{2}}{x_{m+2i+1}}}.$$
We then have
\begin{align*}
J_{m+2,n}&=J_{m,n}=J_0=\frac{\sum_{i=0}^{n-1}{x_i} + A}{\prod_{i=0}^{\frac{n-3}{2}}{x_{2i+1}}}, \qquad 2\mid m \\
J_{m+2,n}&=J_{m,n}=J_1=\frac{\sum_{0\leq i<j\leq n-1}x_ix_j + A\sum_{i=0}^{n-1}x_i + B}{\prod_{i=1}^{\frac{n-1}{2}}{x_{2i}}},\qquad 2\nmid m.
\end{align*}
and the mutilinear recurrence
\begin{align*}
x_{m+n-1}=J_{\langle m \text{ mod } 2\rangle}\cdot\prod_{i=0}^{\frac{n-3}{2}}x_{m+2i+1}-\sum_{i=0}^{n-2}x_m-A,\qquad m=0,1,2,3\ldots
\end{align*}

\section{Acknowledgements}
This research was conducted at the 2013 summer REU (Research Experience for Undergraduates) program at the University of Minnesota, Twin Cities, and was supported by NSF grants DMS-1067183 and DMS-1148634. We would like to thank Professors Dennis Stanton, Gregg Musiker, Joel Lewis, and especially our mentor Pavlo Pylyavskyy, who directed the program, for their advice and support throughout this project. We would also like to thank Al Garver, Andrew Hone, and other REU participants for their help in editing this paper.

\bibliographystyle{plain}

\end{document}